\documentclass[a4paper,10pt]{amsart}
\usepackage{graphicx}
\pdfoutput=1
\usepackage[colorlinks=false]{hyperref}

\newtheorem{theorem}{Theorem}[section]
\newtheorem{lemma}[theorem]{Lemma}

\newtheorem{remark}[theorem]{Remark}

\numberwithin{equation}{section}

\allowdisplaybreaks

{%

\begin{enumerate}}%
{\end{enumerate}}

\newcommand{\norm}[1]{\left\Vert#1\right\Vert}
\newcommand{\order}[1]{\mathcal{O}\left(#1\right)}
\newcommand{\abs}[1]{\left|#1\right|}
\newcommand{\eps}{\varepsilon}
\newcommand{\R}{\mathbb{R}}
\newcommand{\N}{\mathbb{N}}
\newcommand{\Z}{\mathbb{Z}}

\newcommand{\seq}[1]{\left\{#1\right\}}

\newcommand{\Dx}{{\Delta x}}
\newcommand{\Dt}{{\Delta t}}

\newcommand{\Dp}{D_{+}}
\newcommand{\Dm}{D_{-}}
\newcommand{\Dpm}{D_{\pm}}
\newcommand{\Dtp}{D^t_{+}}
\newcommand{\ave}[1]{\overline{#1}}

\newcommand{\dott}{\,\cdot\,}
\newcommand{\calA}{\mathcal A}
\newcommand{\romnum}[1]{\MakeUppercase{\romannumeral #1}}

\DeclareMathOperator*{\csch}{csch}
\DeclareMathOperator*{\sech}{sech}
\DeclareMathOperator*{\Lip}{Lip}

\begin{document}

\title[Finite difference schemes for the KdV equation]{Convergence of a fully
  discrete finite difference scheme for the Korteweg--de Vries equation}

\author[Holden]{Helge Holden}
\address[Holden]{\newline
    Department of Mathematical Sciences,
    Norwegian University of Science and Technology,
    NO--7491 Trondheim, Norway,\newline
{\rm and} \newline
  Centre of Mathematics for Applications, 
University of Oslo,
  P.O.\ Box 1053, Blindern,
  NO--0316 Oslo, Norway }
\email[]{holden@math.ntnu.no}
\urladdr{www.math.ntnu.no/\~holden}

\author[Koley]{Ujjwal Koley} \address[Koley] {\newline  
   Institut f\"{u}r Mathematik,  \newline
   Julius-Maximilians-Universit\"{a}t W\"{u}rzburg,   
\newline Campus Hubland Nord, Emil-Fischer-Strasse 30, \newline DE--97074,
W\"{u}rzburg, Germany.} 
\email[]{toujjwal@gmail.com}

\author[Risebro]{Nils Henrik Risebro}
\address[Risebro]{\newline
  Centre of Mathematics for Applications, 
University of Oslo,
  P.O.\ Box 1053, Blindern,
  NO--0316 Oslo, Norway,\newline
{\rm and} \newline
Seminar for Applied Mathematics (SAM)  \newline
        ETH Zentrum, 
\newline HG G 57.2, R\"amistrasse 101, CH--8092 Z\"urich, Switzerland.}       

\email[]{nilshr@math.uio.no}
\urladdr{www.math.uio.no/\~{}nilshr}

\date{\today}

\thanks{Supported in part by the Research Council of Norway.}

\begin{abstract}
We prove convergence of a fully discrete finite difference scheme for
the Korteweg--de Vries equation. Both the decaying case on the full
line and the periodic case are considered.
If the initial data $u|_{t=0}=u_0$ is of high regularity, $u_0\in H^3(\R)$, the
scheme is shown to converge to a classical solution,
and if the regularity of the initial data is smaller, $u_0\in
L^2(\R)$, then the scheme converges strongly in $L^2(0,T;L^2_{\mathrm{loc}}(\R))$ to a
weak solution.
\end{abstract}

\maketitle

\section{Introduction}
\label{sec:intro}
The Korteweg--de Vries (KdV among friends and foes) equation, which reads 
\begin{equation}
  \label{eq:kdv0}
   u_t + uu_x + u_{xxx}= 0,
\end{equation}
has been studied extensively since its first analysis in 1895 by
Korteweg and de Vries.  Apart from  applications as a model for
shallow water waves, the KdV equation has maintained a pivotal role in
several branches of mathematics.  We here focus on the derivation of
convergent numerical methods for the initial value problem where the
equation \eqref{eq:kdv0} is augmented by initial data $u|_{t=0}=u_0$.
The problem of analyzing convergent numerical schemes is of course
intimately connected with the mathematical properties of the Cauchy
problem for the KdV equation, which has undergone a tremendous
development the last two decades, see, e.g.,
\cite{Tao:2006,LinaresPonce:2009} and the references therein. We will not be able to discuss this
literature here, but only refer to the parts that are pertinent to the
current paper.

In this paper we analyze the implicit finite difference scheme
\begin{equation}
  \label{eq:sjobergexp1AA}
  u^{n+1}_j = \ave{u}^n_j - \Dt\, \ave{u}^n_j D u^n_j - \Dt\, \Dp^2 \Dm u^{n+1}_j, 
  \quad n\in\N_0, \, j\in \Z,
\end{equation}
where $u^n_j\approx u(j\Dx, n\Dt)$, and $\Dx,\Dt$ are small discretization parameters. Furthermore, $D$ and $\Dpm$ denote symmetric and forward/backward (spatial)  finite differences, respectively, and $\ave{u}$ denotes a spatial average.  Two results are proven, both for the full line and the periodic case: 
(1) In the case of initial data $u_0\in H^3(\R)$, we show (see Theorem \ref{thm:H3convergence} and Remark \ref{rem:large}) that the approximation  \eqref{eq:sjobergexp1AA} converges uniformly as $\Dx\to 0$ with  $\Dt=\order{\Dx^2}$  in
 $C(\R\times [0,\bar{T}])$ for any positive $\bar{T}$  to the unique solution of the KdV equation.
(2) When the initial data  $u_0\in L^2(\R)$, we prove that (see Theorem \ref{theo:main_fully}) that the approximation converges strongly as $\Dx\to 0$ with 
  $\Dt=\order{\Dx^2}$  in $L^2(0,T;L^2_{\mathrm{loc}}(\R))$ to a weak solution of the KdV equation.

\medskip
An interesting fact, and rarely referred to in the current literature,
is that the first mathematical proof of existence and uniqueness of
solutions of the KdV equation, was accomplished by Sj\"oberg
\cite{Sjoberg:1970} in 1970, using a finite difference approximation very much
in the spirit considered here. His proof is valid for initial data
that are periodic and with square integrable third derivative,
that is, $u_0(x+1)=u_0(x)$ for $x\in[0,1]$ and
$u_0'''\in L^2([0,1])$. Sj\"oberg's uniqueness proof still is the standard one, using
the Gronwall inequality.  His approach is based on a semi-discrete
approximation where one discretizes the spatial variable, thereby
reducing the equation to a system of ordinary differential equations. 
However, we stress that for numerical computations also this set of
ordinary differential equations will have to be discretized in order
to be solved. Thus in order to have a completely satisfactory numerical method, one
seeks a fully discrete scheme that reduces the actual computation
to a solution of a finite set of algebraic equations. This is
accomplished in the present paper, both in the periodic case and on
the full line. 

There has been a number of papers involving the numerical computation
of solutions of the Cauchy problem, starting with the landmark paper
by Zabusky and Kruskal \cite{ZabuskyKruskal:1965}, where they
discovered the permanence of solitons (the term ``solitons'' being 
coined in the 
same paper) for the KdV equation using numerical techniques. However, 
we will here focus on papers that discuss numerical methods \emph{per se}.

A popular numerical approach has been the application of various
spectral methods. Little is known rigorously about the convergence of
these methods.  For a survey and a comparison, see
\cite{NouriSloan:1989}.  See also \cite{KassamTrefethen:2005}.
Multisymplectic schemes have been studied in
\cite{AscherMcLachlan:2005} (see also references therein).  There
exist convergence proofs for finite element methods for the KdV
equation, see
\cite{Winther:1980,ArnoldWinther:1982,BakerDougalisKarakashian:1983,BonaDougalisKarakashian:1986}.
However, the resulting schemes tend to be quite different from finite
difference schemes derived \emph{ab initio}.

The numerical computation of solutions of the KdV equation is rather
capricious. Two competing effects are involved, namely the nonlinear
convective term $u u_x$, which in the context of the Burgers equation
$u_t+u u_x=0$ yields infinite gradients in finite time even for smooth
data, and the linear dispersive term $u_{xxx}$, which in the Airy
equation $u_t+u_{xxx}=0$ produces hard-to-compute dispersive waves,
and these two effects combined makes it difficult to obtain accurate
and fast numerical methods.  Indeed, any initial data for the Burgers
equation that is decreasing in a small neighborhood, will develop
infinite gradients in finite time, while the Airy equation preserves
all Sobolov norms while creating many oscillatory waves.  Most finite
difference schemes will consist of a sum of two terms, one
discretizing the convective term and one discretizing the dispersive
term. These two effects will have to balance each other, as it is
known that the KdV equation itself keeps the Sobolov norm $H^s$
bounded; from \cite{BonaSmith:1975} we know that if $u_0\in H^s(\R)$,
with $s\ge 3$, then the solution satisfies $\norm{u(t)}_{H^s(\R)}\le
C_{T,u_0}$ for $t\in[0,T]$.  This dichotomy between these two effects
is brought to the forefront in the method of operator splitting. Here
the two equations, the Burgers equations and the Airy equation, are
solved sequentially for a small time step. This procedure is iterated,
and as the time step converges to zero, the approximation converges to
the actual solution.  In the KdV context operator splitting was
introduced by Tappert \cite{Tappert:1974}, a Lax--Wendroff theorem was
proved in \cite{HoldenKarlsenRisebro:1999}, and convergence of the
operator splitting technique proved in
\cite{HoldenKarlsenRisebroTao:2009,
  HoldenLubichRisebro:2011,HoldenKarlsenKarper:2011,Holden:book}. Our
approach here is a finite difference method which can also be viewed 
as an operator splitting method.

 Recently, a semi-discrete scheme for the generalized KdV equation was shown to
converge in $L^2_{\mathrm{loc}}$ for initial data in $L^2$
\cite{AmorimFigueira}. However, the scheme analyzed here,
which in contrast to the scheme in \cite{AmorimFigueira}, does not
involve an explicit fourth order stabilizing term, and we show 
convergence  for non-smooth initial data. 

The rest of this paper is organized as follows: In
Section~\ref{sec:scheme} we present the necessary notation and define
the numerical scheme. In Section~\ref{sec:smoothdata} we show the
convergence of the scheme for initial data in $H^3(\R)$, while in
Section~\ref{sec:L2} we show the convergence to a weak solution if the
initial data is in $L^2(\R)$. In Section~\ref{sec:numex} we exhibit
some numerical experiments showing the convergence.

\section{The scheme}
\label{sec:scheme}
We start by introducing the necessary notation.  Derivatives will be
approximated by finite differences, and the basic quantities are as
follows.  For any function $p\colon\R\to \R$ we set
$$
D_{\pm} p(x)=\pm \frac1{\Dx}\big(p(x\pm \Dx)-p(x)\big),
\ \text{and}\ D = \frac{1}{2}\left(\Dp + \Dm \right)
$$
for some (small) positive number $\Dx$.  If we introduce the average
$$
\ave{p}(x) = \frac{1}{2}\left(p(x+\Dx)+p(x-\Dx)\right),
$$
we find the Leibniz rule as
\begin{align*}
D(pq)&= \ave{p} Dq+\ave{q} Dp, \\ 
\Dpm (pq)&=S^\pm p \Dpm q+q\Dpm p=S^\pm q \Dpm p+p\Dpm q.
\end{align*}
Here we have defined the shift operator
$$
S^\pm p(x)=p(x\pm\Dx).
$$
We discretize the real axis using $\Dx$ and set $x_j=j\Dx$ for $j\in \Z$. For  a given
function $p$ we define $p_j=p(x_j)$.  We will consider functions in $\ell^2$ with the usual inner product and norm
\begin{equation*}
(p, q)= \Dx\sum_{j\in \Z} p_j q_j, \quad \norm{p}= (p,p)^{1/2}, \qquad
p,q\in \ell^2. 
\end{equation*}  
In the periodic case with period $J$ the sum over $\Z$ is replaced by a finite sum $j=0,\dots,J-1$.
Observe that
\begin{equation*}
 \norm{p}_\infty= \max_{j\in\Z}\abs{p_j}\le \frac1{\Dx^{1/2}  } \norm{p}.
\end{equation*} 
The various difference operators enjoy the following properties:
\begin{equation*}
(p, D_{\pm}q)=-(D_{\mp}p,q), \quad (p,Dq)=-(Dp,q), \qquad p,q\in \ell^2.
\end{equation*}  
Further useful properties include
\begin{equation}
  \label{eq:useful}
  \begin{aligned}
    (u,\Dp^2\Dm u)&= \frac12 (u,\Dm \Dp^2u)-\frac12(u,\Dp \Dm^2u)\\
    &=\frac12 \big(u,(\Dm \Dp^2-\Dp\Dm^2)u\big)\\
    &=\frac12 \big(u,\Dm \Dp(\Dm-\Dp)u\big)\\
    &=\frac{\Dx}{2}\norm{\Dp\Dm u}^2
  \end{aligned}
\end{equation}
since $(u,\Dp \Dm^2 u)=-(u,\Dm \Dp^2u)$ (because $(u,v)=(v,u)$) in the first line, and
\begin{equation*}
\Dm-\Dp=-\Dx\,\Dp\Dm=-\Dx\, \Dm\Dp.
\end{equation*}
We also need to discretize in the time direction. Introduce
(a small) time step $\Dt>0$, and use the notation
\begin{equation*}
\Dtp p(t)=\frac1{\Dt}\big(p(t+\Dt)-p(t)\big),
\end{equation*}
for any function $p\colon[0,T]\to \R$. Write  $t_n=n\Dt$ for 
$n\in\N_0=\N\cup\{0\}$. A fully discrete grid
function is a function $u_\Dx\colon \Dt\, \N_0 \to \R^\Z$, and we write
 $u_\Dx(x_j,t_n)=u^n_j$. (A CFL condition will enforce a relationship
 between $\Dx$ and $\Dt$, and hence we only use $\Dx$ in the notation.)


We propose the following implicit scheme to generate approximate
solutions to the KdV equation \eqref{eq:kdv0}
\begin{equation}
  \label{eq:sjobergexp1}
  u^{n+1}_j = \ave{u}^n_j - \Dt\, \ave{u}^n_j D u^n_j - \Dt\, \Dp^2 \Dm u^{n+1}_j, 
  \quad n\in\N_0, \, j\in \Z.
\end{equation}
For the initial data we have
$$
u^0_j=u_0(x_j), \quad j\in \Z.
$$
\begin{remark}
  This scheme can be reformulated as an operator splitting scheme as
  follows. Set 
  \begin{equation*}
    u^{n+1/2}_j = \frac12\left(u^n_{j+1}+u^n_{j-1}\right) -
    \frac{\Dt}{2\Dx}
    \left(\frac12 \left(u^n_{j+1}\right)^2 - \frac12
      \left(u^n_{j-1}\right)^2\right),
  \end{equation*}
  i.e., $u^{n+1/2}$ is solution operator of the Lax--Friedrichs scheme
  for Burgers' equation, applied to $u^n$. Then
  \begin{equation*}
    \frac{u^{n+1}-u^{n+1/2}}{\Dt}= - \Dp^2\Dm u^{n+1},
  \end{equation*}
  i.e., $u^{n+1}$ is the approximate solution operator of a 
  first-order implicit scheme for Airy's
  equation $u_t+u_{xxx}=0$. If we write these two approximate
  solution operators as $S^{B}_{\Dt}$, and $S^{A}_\Dt$, respectively,
  the update formula \eqref{eq:sjobergexp1} reads
  \begin{equation*}
    u^{n+1}=\left(S_\Dt^A\circ S_\Dt^B\right) u^n.
  \end{equation*}
  The convergence of this type of operator splitting using exact
  solution operators have been shown in
  \cite{HoldenKarlsenRisebroTao:2009,HoldenLubichRisebro:2011}, with
  severe restrictions on the initial data. The results in this paper
  can be viewed as a convergence result for operator splitting using
  approximate operators with less restrictions on the initial data,
  but with specified ratios between the temporal and spatial
  discretizations (CFL-like conditions).
\end{remark}
\section{Convergence for smooth initial data}\label{sec:smoothdata}
To show that the implicit scheme can be solved with respect to
$u^{n+1}_j$, we proceed as follows: Write the scheme as
\begin{equation*}
(1+\Dt \Dp^2 \Dm)u^{n+1}_j= \ave{u}^n_j - \Dt\, \ave{u}^n_j D u^n_j
\end{equation*}
and hence
\begin{align*}
((1+\Dt \Dp^2 \Dm)u^{n+1}, u^{n+1})&=\norm{u^{n+1}}^2
+\Dt(\Dp^2 \Dm u^{n+1},u^{n+1}) \\
&=\norm{u^{n+1}}^2
+\frac12\Dt\Dx\norm{\Dp \Dm u^{n+1}}^2 \\
&\ge\norm{u^{n+1}}^2, 
\end{align*}
thus
\begin{equation*}
  \norm{u^{n+1}}  \le \norm{(1+\Dt \Dp^2 \Dm)u^{n+1}}= \norm{\ave{u}^n -
    \Dt\, \ave{u}^n D u^n}. 
\end{equation*}
The fundamental stability lemma reads as follows.
\begin{lemma}\label{lem:stability1}
Let $u^n_j$ be a solution of the difference scheme
\eqref{eq:sjobergexp1}. Then the following estimate holds
\begin{multline}
  \label{eq:L2}
  \norm{u^{n+1}}^2 + \Dt \Dx^{1/2}\Big(\Dx\lambda \norm{\Dp^2\Dm u^{n+1}}^2 
  +\Dx^{1/2}  \norm{\Dp\Dm u^{n+1}}^2 +
  \frac{\delta}{\lambda}\norm{Du^n}^2 \Big)\\ 
   \le \norm{u^n}^2, 
\end{multline}
provided the CFL condition
\begin{equation}
  \label{eq:CFL1a}
  \lambda\norm{u^0} \left( \frac{1}{3} +\frac12
    \lambda \norm{u^0} \right) 
  < \frac{1-\delta}{2},   \quad \delta\in(0,1), 
\end{equation}
holds where  $\lambda=\Dt/\Dx^{3/2}$.
\end{lemma}
\begin{proof}
For the moment we drop the indices $j$ and $n$ from our notation, and use the
notation $u$ for $u_j^n$ where $j$ and $n$ are fixed. We first study the
``Burgers'' term $\Dt\, \ave{u}Du$.
Let $u$ be a grid function and set 
\begin{equation}
  w=\ave{u}-\Dt\, \ave{u}Du.\label{eq:wdef}
\end{equation}
If the timestep $\Dt$ satisfies \eqref{eq:CFL1a} then we have the
following ``cell entropy'' inequality
\begin{equation}
  \label{eq:L2cell}
  \frac{1}{2}w^2 \le \frac{1}{2}\ave{u^2} - \frac{\Dt}{3}D u^3 - 
  \delta \frac{\Dx^2}{2}\left(D u\right)^2, \quad \delta\in(0,1).
\end{equation}
To prove this we multiply \eqref{eq:wdef} by $\ave{u}$ to find
\begin{align*}
  \frac{1}{2}w^2 &= \frac{1}{2}\ave{u}^2 - \Dt \frac{1}{2}\ave{u}D u^2 +
  \frac{1}{2}\left(w-\ave{u}\right)^2\\
  &= \frac{1}{2}\ave{u^2} -  \Dt \frac{1}{2}\ave{u}D u^2 +
  \frac12\Dt^2 \ave{u}^2\left(Du\right)^2 +
  \frac{1}{2}\left(\ave{u}^2 - \ave{u^2}\right).
\end{align*}
Now we have that 
\begin{equation*}
  \frac{1}{4}(a+b)\left(a^2-b^2\right) = 
  \frac{1}{3}\left(a^3-b^3\right) - \frac{1}{12} (a-b)^3,
\end{equation*}
and 
\begin{equation*}
  \frac{1}{4}(a+b)^2 - \frac{1}{2}\left(a^2+b^2\right) = 
  -\frac{1}{4}(a-b)^2.
\end{equation*}
For a grid function, this implies 
\begin{align*}
  \ave{u}D u^2 &= \frac{2}{3}D u^3 -
  \frac{2\Dx^2}{3}\left(Du\right)^3,\\
  \ave{u}^2 - \ave{u^2} &= - \Dx^2\left(Du\right)^2.
\end{align*}
Therefore
\begin{align*}
  \frac{1}{2}w^2 &= \frac{1}{2}\ave{u^2} - \frac{\Dt}{3}D u^3
  +\frac12\Dt^2\ave{u}^2 \left(Du\right)^2 + \frac{\Dt\Dx^2}{3}\left(D
    u\right)^3 -
  \frac{\Dx^2}{2}\left(Du\right)^2\\
  &=\frac{1}{2}\ave{u^2} - \frac{\Dt}{3}D u^3 -
  \delta \frac{\Dx^2}{2} \left(D u\right)^2 \\
  &\qquad + \Dx^2\left(Du\right)^2 \left(\frac{\lambda}{3}\Dx^{3/2} Du
    + \frac12\Dx\lambda^2 \ave{u}^2
    - \frac{1-\delta}{2}\right) \\
  &\le\frac{1}{2}\ave{u^2} - \frac{\Dt}{3}D u^3 -
  \delta \frac{\Dx^2}{2} \left(D u\right)^2 \\
  &\qquad + \Dx^2\left(Du\right)^2 \left(\frac{\lambda}{3}\Dx^{1/2}
    \norm{u}_\infty + \frac12\Dx\lambda^2 \ave{u}^2
    - \frac{1-\delta}{2}\right) \\
  &\le\frac{1}{2}\ave{u^2} - \frac{\Dt}{3}D u^3 -
  \delta \frac{\Dx^2}{2} \left(D u\right)^2 \\
  &\qquad + \Dx^2\left(Du\right)^2
  \underbrace{\left(\frac{\lambda}{3} \norm{u} +
      \frac12 \lambda^2 \norm{u}^2
      - \frac{1-\delta}{2}\right)}_A\\
  &\le\frac{1}{2}\ave{u^2} - \frac{\Dt}{3}D u^3 - \delta
  \frac{\Dx^2}{2} \left(D u\right)^2, \quad \delta\in(0,1),
\end{align*}
where we have employed that $A<0$ since $\lambda$ satisfies the CFL
condition \eqref{eq:CFL1a}. Estimate \eqref{eq:L2cell} follows.

Summing \eqref{eq:L2cell} over $j$ we get 
\begin{equation}
  \label{eq:L2burg}
  \norm{w}^2 +\delta\Dx^2\norm{Du}^2 \le \norm{u}^2.   
\end{equation}
Next we study the full difference scheme by adding the ``Airy term''
$\Dt\, \Dp^2\Dm  u^{n+1}_j$. Thus the full difference scheme 
\eqref{eq:sjobergexp1} can be written
$$
v=w-\Dt\,\Dp^2\Dm v.
$$
Writing this as $w=v+\Dt\,\Dp^2\Dm    v$, we square it and  sum over $j$
to get
\begin{equation}
  \label{eq:345}
  \begin{aligned}
    \norm{w}^2&=\norm{v}^2+2\Dt (v,\Dp^2\Dm v)+ \Dt^2 \norm{\Dp^2\Dm  v}^2\\
    &=\norm{v}^2+\Dt\Dx \norm{\Dm\Dp v}^2+ \Dt^2 \norm{\Dp^2\Dm  v}^2,
  \end{aligned}
\end{equation}
using  the  identity \eqref{eq:useful}.

For the function $u^n$ this means that 
\begin{equation}
  \label{eq:unL2*}
  \begin{aligned}
    \norm{w}^2 &= \norm{u^{n+1}}^2 \\
    &\qquad + \Dt \Dx^{1/2}\left(\Dx\lambda
      \norm{\Dp^2\Dm u^{n+1}}^2 +
      \Dx^{1/2} \norm{\Dp\Dm u^{n+1}}^2  \right) \\
    & \le \norm{u^n}^2 -\delta\Dx^2\norm{Du}^2,
  \end{aligned}
\end{equation}
using \eqref{eq:L2burg}. This implies
\begin{multline}
  \label{eq:unL2}
  \norm{u^{n+1}}^2 + \Dt \Dx^{1/2}\big(\Dx\lambda \norm{\Dp^2\Dm    u^{n+1}}^2 \\
  +\Dx^{1/2}  \norm{\Dp\Dm u^{n+1}}^2 + \frac{\delta}{\lambda}\norm{Du^n}^2 \big) 
  \le
  \norm{u^n}^2. 
\end{multline}
\end{proof}
Next, we consider what corresponds to the finite difference scheme
satisfied by the time derivative of the original scheme.
\begin{lemma}
  Let $u^n_j$ be a solution of the difference scheme
  \eqref{eq:sjobergexp1}. Then the following estimate holds
  \begin{equation}
    \label{eq:vnest_lemma}
    \begin{aligned}
      \norm{\Dtp u^{n}}^2 &+\Dt^2\norm{\Dp^2\Dm \Dtp u^{n}}^2\\
      &+
      \Dt\Dx\norm{\Dp\Dm \Dtp u^{n}}^2 +
      \tilde\delta\Dx^2 \norm{D\Dtp u^{n-1} }^2 
      \\
      &\qquad\qquad\qquad\qquad\qquad  \le  \norm{\Dtp u^{n-1} }^2\left(1 +
      3 \Dt \norm{Du^n}_\infty\right),
    \end{aligned}
  \end{equation}
  provided $\Dt$ is chosen such that
  \begin{equation}
    \label{eq:cfl2_lemma}
     6\norm{u_0}^2\lambda^2 +
    \norm{u_0} \lambda <
    \frac{1-\tilde\delta}{2}, \qquad \tilde\delta\in(0,1).
  \end{equation}
\end{lemma}
\begin{proof}
  Introduce
  \begin{equation*}
    \alpha^n = \Dtp u^{n-1}=\frac{1}{\Dt}(u^n-u^{n-1}), \quad n\in\N.
  \end{equation*}
  Using \eqref{eq:sjobergexp1} we see that this grid function satisfies
  \begin{align}
    \alpha^{n+1}&= \ave{\alpha}^n - \Dt \left(\ave{\alpha}^n D u^n +
      \ave{u}^{n-1} D\ \alpha^n
    \right) + \Dt^2 \ave{\alpha}^n D\alpha^n - \Dt \Dp^2\Dm \alpha^{n+1} \notag\\
    &= \ave{\alpha}^n - \Dt \, D(u^n \alpha^n)+ \Dt^2 \ave{\alpha}^n
    D\alpha^n - \Dt \Dp^2\Dm \alpha^{n+1}, \quad n\in
    \N.  \label{eq:vndef}
  \end{align}
 Introduce
  \begin{equation}\label{eq:vndef10}
    \beta=\ave{\alpha} - \Dt \, D(u\alpha) + \frac{\Dt^2}{2}D\alpha^2,
  \end{equation}
  which means that \eqref{eq:vndef} can be written as
  \begin{equation}\label{eq:vndef1}
    \alpha^{n+1}= \beta - \Dt \Dp^2\Dm \alpha^{n+1}.
  \end{equation}
  We proceed as before and square \eqref{eq:vndef10} to find
  \begin{align*}
    \frac{1}{2}\beta^2 &= \frac{1}{2}\ave{\alpha^2} + \frac{\Dt^2}{2}
    \left( D(u\alpha) - \frac{\Dt}{2} D\alpha^2 \right)^2 
    \\ &\qquad -
    \Dt \left( \ave{u}\,\ave{\alpha}D\alpha + \ave{\alpha}^2 D
      u\right) + \Dt^2 \ave{\alpha}^2 D\alpha 
      -\frac{\Dx^2}{2}\left(D\alpha\right)^2.
  \end{align*}
  We have that
  \begin{align*}
    \frac{1}{2} \left( D(u\alpha) - \frac{\Dt}{2} D\alpha^2 \right)^2
    &\le
    \left(D(u\alpha)\right)^2 + \Dt^2 \ave{\alpha}^2 \left(D\alpha\right)^2 \\
    &\le 2 \ave{u}^2 \left(D\alpha\right)^2 + 2\ave{\alpha}^2
    \left(Du\right)^2 +
    \Dt^2 \ave{\alpha}^2 \left(D\alpha\right)^2,\\
    \ave{\alpha}^2D\alpha &= \frac{1}{3}D\alpha^3 -
    \frac{\Dx^2}{3}\left(D\alpha\right)^3,\\
    \ave{u}\,\ave{\alpha} D\alpha + \ave{\alpha}^2 Du &= \frac{1}{2}
    D\left(u\alpha^2\right) + \frac{1}{2} \ave{\alpha}^2 Du -
    \frac{\Dx^2}{2} \left(D\alpha\right)^2 Du.
  \end{align*}
  Using this
  \begin{equation}
    \begin{aligned}
      \frac{1}{2}\beta^2 &\le \frac{1}{2}\ave{\alpha^2} -
      \frac{\Dt}{2} D\left(u \alpha^2 - \frac{2\Dt}{3} \alpha^3\right)
      - \frac{\Dt}{2}\ave{\alpha}^2 Du +\frac{\Dt\Dx^2}{2}
      \left(D\alpha\right)^2 Du \\
      &\quad + \Dt^2\left( 2 \ave{u}^2 \left(D\alpha\right)^2 +
        2\ave{\alpha}^2 \left(Du\right)^2 + \Dt^2 \ave{\alpha}^2
        \left(D\alpha\right)^2
        - \frac{\Dx^2}{3} \left(D\alpha\right)^3\right) \\
      &\quad - \frac{\Dx^2}{2}\left(D\alpha\right)^2.
    \end{aligned}\label{eq:est1}
  \end{equation}
  Now we must balance the positive terms with $\Dx^2(D\alpha)^2$. To
  this end we estimate
  \begin{align*}
    \Dx^2( D\alpha)^2 &\le \ave{\alpha^2},\\
    \Dt^2 \ave{\alpha}^2(Du)^2 &\le \lambda\Dx^{1/2} \, \Dt
    \norm{u^n}_\infty \ave{\alpha^2}  \abs{Du^n},\\
    \Dt^2 \ave{\alpha}^2 &\le 2\norm{u^n}_\infty^2 + 2 \norm{u^{n-1}}_\infty^2,\\
    \Dx^2 \abs{D\alpha} &\le
    \frac{1}{\lambda\Dx^{1/2}}\left(\norm{u^n}_\infty +
      \norm{u^{n-1}}_\infty\right).
  \end{align*}
  Using these in \eqref{eq:est1} we find
  \begin{align*}
    \frac{1}{2}\beta^2 &\le \frac{1}{2}\ave{\alpha^2} - \frac{\Dt}{2}
    D\left(u \alpha^2 - \frac{2\Dt}{3} \alpha^3\right) +
    (\frac{\Dt}{2}+\frac{\Dt}{2}) \ave{\alpha^2}\abs{Du}
    +2\Dt^2\,\ave{\alpha}^2(Du)^2
    \\
    &\quad + \lambda^2 \Dx^3 \left(D\alpha\right)^2 \left(
      2\ave{u}^2+ 2\norm{u^n}_\infty^2 + 2 \norm{u^{n-1}}_\infty^2
      +\frac{1}{3\lambda\Dx^{1/2}}\left(\norm{u^n}_\infty +
        \norm{u^{n-1}}_\infty\right)
    \right) \\
    &\quad - \frac{\Dx^2}{2} \left(D\alpha\right)^2
    \\
    &\le \frac{1}{2}\ave{\alpha^2} - \frac{\Dt}{2} D\left(u\alpha^2 -
      \frac{2\Dt}{3} \alpha^3\right) + \Dt\left(1+\lambda\Dx^{1/2}
      \,\norm{u^n}_\infty\right) \ave{\alpha^2}\abs{Du^n}
    \\
    &\quad + \lambda^2 \Dx^2 \left(D\alpha\right)^2 \left(2
      \Dx\left(2\norm{u^n}_\infty^2 +  \norm{u^{n-1}}_\infty^2\right)
      +\frac{\Dx^{1/2}}{3\lambda}\left(\norm{u^n}_\infty
        +\norm{u^{n-1}}_\infty\right)\right)\\
    &\quad
    - \frac{\Dx^2}{2} \left(D\alpha\right)^2 \\
    &\le \frac{1}{2}\ave{\alpha^2} - \frac{\Dt}{2}
    D\left(u\alpha^2-\frac{2\Dt}{3} \alpha^3\right) +
    \Dt(1+\lambda\norm{u_0})\ave{\alpha^2}\abs{Du^n}
    \\
    &\quad + \lambda^2 \Dx^2 \left(D\alpha\right)^2
    \left(6\norm{u_0}^2 +\frac{2}{3\lambda}\norm{u_0} -
      \frac{1-\tilde\delta}{2\lambda^2}\right) -
    \tilde\delta\frac{\Dx^2}{2} \left(D\alpha\right)^2
    \\
    &\le \frac{1}{2}\ave{\alpha^2} - \frac{\Dt}{2} D\left(u \alpha^2 -
      \frac{2\Dt}{3} \alpha^3\right)
    + \Dt\frac{3-\tilde\delta}{2} \,\ave{\alpha^2}\abs{Du^n}\\
    &\quad + \Dx^2 \left(D\alpha\right)^2
    \left(6\norm{u_0}^2\lambda^2 +\norm{u_0}\lambda -
      \frac{1-\tilde\delta}{2}\right) - \tilde\delta\frac{\Dx^2}{2}
    \left(D\alpha\right)^2, \quad \tilde\delta\in(0,1).
  \end{align*}
  Here we have enforced the CFL condition \eqref{eq:cfl2_lemma}, which in particular implies that
  $\norm{u_0}\lambda \le (1-\tilde\delta)/2$. To simplify the numerical expressions, we have employed $\frac23\le 1$.
  Now we 
  multiply with $\Dx$ and sum over $j$ to obtain
  \begin{equation}\label{eq:vnest1}
    \frac{1}{2}\norm{\beta}^2 +\tilde\delta\frac{\Dx^2}{2}
    \norm{D\alpha}^2\le  
    \frac{1}{2}\norm{\alpha}^2 + \frac{3-\tilde{\delta}}{2}\Dt
    \norm{Du^n}_\infty \norm{\alpha}^2.
  \end{equation}
  Writing equation \eqref{eq:vndef1} as
  \begin{equation*}
    \beta^2= \left(\alpha^{n+1}+ \Dt \Dp^2\Dm \alpha^{n+1}\right)^2,
  \end{equation*}
  we find
  \begin{align*}
    \norm{\beta}^2&= \norm{\alpha^{n+1}}^2+2\Dt (\alpha^{n+1},\Dp^2\Dm
    \alpha^{n+1})
    +\Dt^2\norm{\Dp^2\Dm \alpha^{n+1}}^2\\
    &=\norm{\alpha^{n+1}}^2+\Dt\Dx \norm{\Dp\Dm \alpha^{n+1}}^2
    +\Dt^2\norm{\Dp^2\Dm \alpha^{n+1}}^2.
  \end{align*}
  Combining this with \eqref{eq:vnest1} we find
  \begin{multline}
    \label{eq:vnest}
    \norm{\alpha^{n+1}}^2 +\Dt\Dx\norm{\Dp\Dm\alpha^{n+1}}^2
    +\Dt^2\norm{\Dp^2\Dm\alpha^{n+1}}^2\\
    +\tilde\delta\Dx^2\norm{D\alpha^n}^2 \le 
    \left(1+3\Dt \norm{Du^n}_\infty
    \right)\norm{\alpha^n}^2 .
  \end{multline}
\end{proof}
At this point we recall the inequality 
(cf.~Lemma~\ref{lem:SobolevDisc}):
\begin{equation}
  \norm{Du}_\infty \le \eps \norm{\Dp^2 \Dm u} +
  C(\eps)\norm{u},\label{eq:sjo1}
\end{equation}
where $\eps$ is any constant, and $C(\eps)$ is another
constant depending on $\eps$. 

The definition of $u^n$,
\eqref{eq:sjobergexp1}, can be rewritten
\begin{equation}
  \label{eq:sjobergexp2}
  \alpha^{n+1}=\Dtp u^n = \frac{1}{2\mu}\Dp\Dm u^n - \ave{u}^n Du^n -
  \Dp^2\Dm u^{n+1},
\end{equation}
where $\mu=\Dt/\Dx^2=\lambda/\Dx^{1/2}$.
Therefore (using Lemma~\ref{lem:stability1} in the second estimate)
\begin{align*}
  \norm{\Dp^2\Dm u^{n+1}}&\le\norm{\alpha^{n+1}} +
  \norm{\ave{u}^{n}Du^n} + \frac{1}{2\mu} \norm{\Dp\Dm
    u^n}
  \\
  &\le \norm{\alpha^{n+1}} + \norm{Du^n}_\infty\norm{u_0}
  +\frac{1}{2\mu}\left(\eps \norm{\Dp^2\Dm u^n} + 
    C(\eps) \norm{u_0}\right)
  \\
  &\le \norm{\alpha^{n+1}} + 
  \norm{u_0}\left(\eps_1 \norm{\Dp^2\Dm u^n}+
    C(\eps_1)\norm{u_0}\right)\\
  &\qquad \qquad 
  +\frac{1}{2\mu} \eps \norm{\Dp^2\Dm u^n} +
  \frac{1}{2\mu} C(\eps)\norm{u_0}\\
  &\le 
  \norm{\alpha^{n+1}} +
  \left(\eps_1 \norm{u_0} + \frac{\Dx^{1/2}}{2\lambda}\eps\right)
  \norm{\Dp^2\Dm u^n}\\
  &\qquad\qquad + \underbrace{C(\eps_1)\norm{u_0}^2 +
    \frac{\Dx^{1/2}}{2\lambda} C(\eps)\norm{u_0}}_{\calA(\eps_1,\eps)}
  \\
  &\le
  \norm{\alpha^{n+1}} +
  \frac12 \norm{\Dp^2\Dm u^n} + \calA\ \text{ (choosing $\eps_1$ and
    $\eps$ such that this holds)}
  \\
  &= 
  \norm{\alpha^{n+1}} +
  \frac12 \norm{\Dp^2\Dm\left(u^{n+1}-\Dt \alpha^{n+1}\right)} + \calA\\
  &\le 
  \norm{\alpha^{n+1}} +
  \frac12 \norm{\Dp^2\Dm u^{n+1}} + \frac12 \Dt \norm{\Dp^2\Dm \alpha^{n+1}} +
  \calA\\
  &\le 
  \norm{\alpha^{n+1}} +
  \frac12 \norm{\Dp^2\Dm u^{n+1}} + 
  \frac12 \norm{\alpha^n}\left(1+3\Dt\norm{Du^n}_\infty\right)^{1/2}+\calA
  \\
  &\le 
  \norm{\alpha^{n+1}} +
  \frac12 \norm{\Dp^2\Dm u^{n+1}}+
  \frac12 \norm{\alpha^n}\left(1+3\lambda
      \norm{u_0}\right)^{1/2} +\calA,
\end{align*}
where we have used \eqref{eq:vnest} to estimate 
$\Dt\norm{\Dp^2\Dm  \alpha^{n+1}}$. Hence
\begin{equation}
  \label{eq:dpdm2bnd}
  \norm{\Dp^2\Dm u^{n+1}} \le c_0 + c_1 \norm{\alpha^{n+1}} + c_2 \norm{\alpha^{n}},
\end{equation}
for some constants $c_0$, $c_1$ and $c_2$ that are independent of
$\Dx$. Exploiting this and the  inequality (cf.~Lemma \ref{lem:SobolevDisc}) in
\eqref{eq:vnest}, we get
\begin{align*}
  \norm{\alpha^{n+1}}^2 &\le \norm{\alpha^n}^2 + \Dt
  \left(\eps\norm{\Dp^2\Dm u^n} +
    C(\eps)\norm{u^n}\right)\norm{\alpha^n}^2\\
  &\le \norm{\alpha^n}^2 + C\Dt \left(\eps\left(c_0 + c_1 \norm{\alpha^n} +
      c_2\norm{\alpha^{n-1}}\right) +
    C(\eps)\norm{u_0}\right)\norm{\alpha^n}^2. 
\end{align*}
Since $\norm{u^n}$ is bounded by $\norm{u_0}$, 
\begin{equation}
  \label{eq:vdiffest}
  \norm{\alpha^{n+1}}^2 \le \norm{\alpha^n}^2 + \Dt\left(
    d_1\norm{\alpha^n}^2 + d_2 \left(\norm{\alpha^n}^3 
    + \norm{\alpha^n}^2\norm{\alpha^{n-1}}\right)\right),
\end{equation}
for constants $d_1$ and $d_2$ which only depend on $\norm{u_0}$ and
$\lambda$.  Set $a_n=\norm{\alpha^n}^2$, so that
\begin{equation*}
a_{n+1}\le a_n + \Dt \left( d_1 a_n + d_2\left(a_n^{3/2} + a_na_{n-1}^{1/2}
\right)\right).
\end{equation*}
Now let $A=A(t)$ be the solution of the differential
equation 
\begin{equation*}
\frac{d A}{dt} = d_1 A + 2d_2 A^{3/2}, \qquad
A\left(t_1\right)= a_1> 0.
\end{equation*}
This solution has blow-up time 
\begin{equation*}
T^\infty= t_1+\frac{2}{d_1}\ln\left(1+\frac{d_1}{d_2 \sqrt{a_0}}\right).
\end{equation*}
Furthermore, for $t<T^\infty$, $A$ is a convex function of
$t$ (since the second derivative clearly is non-negative). We now claim that for $t_n<T^\infty$,
we have 
\begin{equation*}
a_n\le A(t_n), \quad n\in \N.
\end{equation*}
This  holds for $n=1$ by construction.
Assuming that the claim holds for natural numbers
up to $n$, we get
\begin{align*}
  a_{n+1} &\le A\left(t_n\right) \\
  &\qquad + \Dt \Big( d_1
    A\left(t_n\right)\\
&\qquad\qquad + d_2\left(
      A\left(t_n\right)^{3/2} +
      A\left(t_n\right)
      A\left(t_{n-1}\right)^{1/2} \right)\Big)\\
  &\le A\left(t_n\right) + \Dt \left( d_1
    A\left(t_n\right) + 2d_2 
    A\left(t_n\right)^{3/2}\right)\\
  &\le A(t_{n+1}).
\end{align*}
The last inequality follows from
\begin{align*}
A(t_{n+1})- A\left(t_n\right)&= \int_{t_n}^{t_{n+1}} A'(s) ds\\
&=\int_{t_n}^{t_{n+1}} \left( d_1
    A\left(s\right) + 2d_2 
    A\left(s\right)^{3/2}\right)ds \\
    &\ge \int_{t_n}^{t_{n+1}} 
A\left(t_n\right) + \Dt \left( d_1
    A\left(t_n\right) + 2d_2 
    A\left(t_n\right)^{3/2}\right)ds
\end{align*}
using the monotonicity. 
Hence, for $t\le \bar{T} = T^\infty/2$, $\norm{\alpha^n}\le C$ for some constant
independent of $\Dx$.

Therefore, we can follow Sj\"oberg \cite{Sjoberg:1970} to prove
convergence of the scheme for $t<\bar{T}$. We reason as follows: Let
$u_{\Dx}(x,t)$ be the piecewise bilinear continuous interpolation 
\begin{equation}   \label{eq:bilinearinterp}
  \begin{aligned}
    u_{\Dx}(x,t) &= u^n_j + \left(x-x_j\right)\Dp u^n_j +
    \left(t-t_n\right)\Dtp u^n_j \\
    &\qquad \qquad + \left(x-x_j\right)\left(t-t_n\right)
    \Dtp \Dp u^n_j
  \end{aligned}
\end{equation}
for $(x,t)\in [x_j,x_{j+1})\times [t_n,t_{n+1})$. Observe that
\begin{equation*}
 u_{\Dx}(x_j,t_n) =u^n_j, \qquad j\in\Z, \quad n\in\N_0.
\end{equation*}
Note that $u_\Dx$ is continuous everywhere and differentiable
almost everywhere.

The function $u_\Dx$ satisfies the bounds
\begin{align}
  \norm{u_{\Dx}(\dott,t)}_{L^2(\R)}&\le \norm{u_0}_{L^2(\R)},\label{eq:udxL2}\\
  \norm{\partial_x u_\Dx(\dott,t)}_{L^2(\R)} &\le C,\label{eq:udxH1}\\
  \norm{\partial_t u_\Dx(\dott,t)}_{L^2(\R)} &\le C,\label{eq:udxtL2}\\
  \norm{\partial_{xxx} u_\Dx(\dott,t)}_{L^2(\R)}&\le C,\label{eq:udxH3}
\end{align}
for $t\le \bar{T}$ and for a constant $C$ which is independent of
$\Dx$. The first three bounds have already been shown, to show the
last bound notice that
\begin{equation*}
  \norm{\Dp^2\Dm u^n}\le \norm{\Dtp u^n} +
  \norm{\bar{u}^n}_\infty\norm{Du^n}\le C.
\end{equation*}
The inequality \eqref{eq:udxH3} follows readily from this. 

The bound
on $\partial_t u_\Dx$ also implies that $u_\Dx\in
\Lip([0,\bar{T}];L^2(\R))$. Then an application of the Arzel\`a--Ascoli
theorem using \eqref{eq:udxL2} shows that the set
$\seq{u_\Dx}_{\Dx>0}$ is sequentially compact in $C([0,\bar{T}];L^2(\R))$,
such that there exist a sequence $\seq{u_{\Dx_j}}_{j\in\N}$ which
converges uniformly in $C([0,\bar{T}];L^2(\R))$ to some function
$u$. Then we can apply the Lax--Wendroff like result from
\cite{HoldenKarlsenRisebro:1999} to conclude that $u$ is a weak
solution. 

The bounds \eqref{eq:udxH1}, \eqref{eq:udxtL2}, and \eqref{eq:udxH3}
means that $u$ is actually a strong solution such that \eqref{eq:kdv0}
holds as an $L^2$ identity. Thus the limit $u$ is the unique solution
to the KdV equation taking the initial data $u_0$.


Summing up, we have proved the following theorem:
\begin{theorem}\label{thm:H3convergence}
  Assume that $u_0\in H^3(\R)$.  Then there exists a
  finite time $\bar{T}$, depending only on $\norm{u_0}_{H^3(\R)}$,
  such that for $t\le \bar{T}$, the difference approximations defined
  by \eqref{eq:sjobergexp1} converge uniformly in $C(\R\times [0,\bar{T}])$
  to the unique solution of the KdV equation \eqref{eq:kdv0} as $\Dx\to 0$ with 
  $\Dt=\order{\Dx^2}$. 
\end{theorem}

\begin{remark}\label{rem:large}
  We can now proceed as in \cite{Sjoberg:1970} to conclude the
  existence of a solution for all time: We know that the size of the
  interval of existence $[0,\bar{T}]$ only depends on the $H^3$ norm
  of the initial data $u_0$. But the exact solution of the KdV
  equation preserves this norm, thus we can define the approximations
  in an interval $[\bar{T},2\bar{T}]$, starting from the initial value
  \begin{equation*}
    u^0_j = \frac{1}{\Dx}\int_{I_j} \lim_{\Dx\to 0}u_\Dx(x,\bar{T})\,dx,
  \end{equation*}
  This can be repeated to conclude that there exists a solution for
  all $t>0$.
\end{remark}

\begin{remark} To keep the presentation fairly short we have only
  provided details in the full line case. However, we note that the
  same proofs apply \textit{mutatis mutandis} also in the periodic
  case. In particular, the Sobolev estimates provided in the appendix
  are based on summation by parts where the decay at infinity is
  replaced by the periodicity, yielding the same results.
\end{remark}

\section{Convergence with $L^2$ initial data}
\label{sec:L2}
In this section we show that the same difference approximation defined by \eqref{eq:sjobergexp1} converges to a solution of the KdV equation in the case of initial data 
$u_0\in  L^2(\R)$. Clearly we cannot use previous estimates, since those estimates depend on the smoothness of initial data. However in \cite{kato}, Kato showed that the solution of the KdV equation possesses an inherent smoothing effect  due to its dispersive character. In particular, such an effect cannot be present in solutions of hyperbolic equations. More precisely, Kato proved that the solution of \eqref{eq:kdv0} satisfies the following inequality:
\begin{equation*}
\Big(\int_{-T}^T\int_{-R}^R \abs{u_x}^2 dxdt \Big)^{1/2}\le C(T,R), \quad T,R>0, 
\end{equation*}
which is the main ingredient in the proof of existence of weak solutions of KdV equation with initial data $u_0\in  L^2(\R)$. Indeed we prove that the approximate solution $u_\Dx$ lies in
\begin{equation*}
W=\{w\in L^2(0,T; H^1(-Q,Q))\mid w_t\in L^{4/3}(0,T; H^{-2}(-Q,Q))\}
\end{equation*}
which suffices to get compactness in $L^2(0,T; L^2(-Q,Q))$  using the Aubin--Simon compactness lemma,  Lemma \ref{lemma:simon}.

Let the function $p$ be defined as $p=\hat{p}*\omega$, where
\begin{align*}
  \hat{p}(x) = \max\seq{1, \min\seq{ (1 + x +R, 1+2R)}},
\end{align*}
and $\omega$ is a symmetric positive function with integral one and
support in $[-1,1]$. We are interested in this function for arbitrary and large values of $R$. All  derivatives of $p$ are
bounded. 
We shall also use that 
\begin{equation*}
  0\le \frac{d}{dx}p(x)= \int_{-R}^R \omega(x-y) dy\le 1.
\end{equation*}
Since $p$ is positive we can define the weighted inner product and
corresponding norms by 
\begin{equation*}
 (u,v)_p = (u, pv) = \Dx \sum_j p_j u_j v_j, \qquad 
 \norm{u}_p^2 = (u,u)_p,
\end{equation*}
where $p_j=p(x_j)$.
Note that $\norm{u}_p^2 \le (1 + 2R) \norm{u}^2$.

Using summation by parts (recall  that $(S^{\pm}u)_j = u_{j\pm1}$), we have 
\begin{align*}
  \left(\Dm\Dp^2 u,u\right)_p &= \left(\Dm\Dp^2 u,up\right)
  \\
  &=-\left(\Dp^2 u, p \Dp u + S^+u \Dp p\right)
  \\
  &=- \left(\Dp\left(\Dp u\right)\Dp u,p\right) - \left(\Dp^2 u,
    S^+u\Dp p\right)
  \\
  &=-\frac12\left(\Dp\left(\Dp u\right)^2,p\right) 
  +\frac{\Dx}{2} \left(\left(\Dp^2 u\right)^2,p \right) - \left(\Dp^2 u,S^+u
    \Dp p\right)
  \\
  &=\frac12 \left(\left(\Dp u\right)^2,\Dm p\right) + \frac{\Dx}2
  \left(\left(\Dp^2 u\right)^2,p \right) +
  \left(\Dp u, \Dm \left(S^+u \Dp p\right)\right)
  \\
  &=\frac12 \left(\left(\Dp u\right)^2,\Dm p\right) + \frac{\Dx}2
  \left(\left(\Dp^2 u\right)^2,p \right) +
  \left(\Dp u, \Dm S^+u \Dp p + u\Dm\Dp p\right)
  \\
  &=\left(\left(\Dp u\right)^2,\frac12 \Dm p + \Dp p\right) + \frac{\Dx}2
  \left(\left(\Dp^2 u\right)^2,p \right) + \left(u \Dp u,\Dm\Dp  p\right)
  \\
  &=\left(\left(\Dp u\right)^2,\frac12 \Dm p + \Dp p\right) + \frac{\Dx}2
  \left(\left(\Dp^2 u\right)^2,p \right) + \frac12 \left(\Dp u^2, \Dm\Dp p\right)\\
  &\qquad \qquad - \frac{\Dx}2 \left(\left(\Dp u\right)^2,\Dp\Dm p\right)
  \\
  &=\left( \left(\Dp u\right)^2, \Dm p+ \frac12 \Dp p\right) + 
  \Dx \left(\left(\Dp^2 u\right)^2,p\right) - \frac12 \left(u^2,\Dp\Dm^2 p\right).
\end{align*}
So we have 
\begin{equation}
  \label{eq:innerprodrule}
  \begin{aligned}
    \left(\Dm\Dp^2 u,u\right)_p &= \left( \left(\Dp u\right)^2, \Dm
      p\right) + \frac12 \left(\left(\Dp u\right)^2,\Dp p\right)\\
    &\qquad \qquad + \frac{\Dx}{2} \norm{\Dp^2 u}_{p}^2 - \frac12
    \left(u^2 ,\Dp\Dm^2 p\right).
  \end{aligned}
\end{equation}

\begin{lemma}\label{lem:h1stability}
Let $u^n_j$ be a solution of the difference scheme
\eqref{eq:sjobergexp1}. Let $N$ be such that $N\Dt = T$, and assume
that the CFL condition \eqref{eq:CFL1a} holds. Then
 \begin{equation}
    \label{eq:uboundh1}
    \begin{aligned}
      \norm{u^{N}}_{p}^2 + &2\Dt \Dx\sum_{n=0}^{N-1}\sum_{\abs{j\Dx}\le R-1}
      \left(\Dp
        u_j^{n+1}\right)^2 \\
      &\qquad \qquad \le \norm{u^0}_p^2 + 4\mu 
      \left(\norm{u^N}^2+\norm{u^0}^2\right) + C,
    \end{aligned}
  \end{equation}
where $\mu=\Dt/\Dx^2=\lambda/\Dx^{1/2}$ and the constant $C$ depends only on  $T$ and $u_0$. In particular,
for any finite number $R$, we have that 
\begin{equation}\label{eq:uderivRbound}
  \Dt\Dx\sum_{n=0}^{N-1} \sum_{\abs{j\Dx}\le R-1} \left(\Dp u^n_j \right)^2 \le C_R,
\end{equation}
where $C_R=C(R,\norm{u_0},T)$.
\end{lemma}
\begin{remark}
We shall see that
this CFL condition is not sufficient to conclude convergence of the
scheme. For that we need $\Dt = \order{\Dx^2}$.
\end{remark}
\begin{proof}
  As before we set
  \begin{equation*}
    w=\ave{u}-\Dt\, \ave{u}Du.
  \end{equation*}
  Set $\lambda=\Dt/\Dx^{3/2}$.  If the timestep $\Dt$ satisfies the
  following CFL condition \eqref{eq:CFL1a} then we can multiply
  \eqref{eq:L2cell} by $p$ to get the ``cell entropy'' inequality
  \begin{equation}
    \label{eq:H1cell}
    \frac{1}{2}pw^2 \le \frac{1}{2}p\ave{u^2} - \frac{\Dt}{3}p D u^3 - 
    \delta \frac{\Dx^2}{2} p \left(D u\right)^2, \quad \delta\in(0,1).
  \end{equation}
  Summing \eqref{eq:H1cell} over $j$ we get
  \begin{equation}
    \label{eq:H1burg}
    \frac12 \norm{w}_p^2 +\delta\frac{\Dx^2}{2}\norm{Du}_p^2 \le 
    \frac12 \norm{u}_p^2 -
    \frac{\Dt}{3} (p, D u^3)+\frac{\Dx^2}{2}\left(u^2,\Dp\Dm p\right).
  \end{equation}
  By \eqref{eq:sob2aD*} we have
  \begin{align*}
    \norm{u\Dp p}_{\infty}&\le \eps \norm{\Dm\left(u\Dp p\right)} +
    C(\eps) \norm{u\Dp p}\\
    &\le \eps\left( \norm{\Dp u \Dp p} + \norm{u\Dm\Dp p}\right) +
    C(\eps) \norm{u\Dp p},
  \end{align*}
  and similarly 
  \begin{equation*}
    \norm{u\Dm p}_\infty \le
    \eps\left(\norm{\Dp u\Dm p} + \norm{S^+u\Dm\Dp p}\right)  + 
    C(\eps) \norm{u\Dm p}.
  \end{equation*}
  We use this to estimate 
  \begin{align*}
    \abs{(p, D u^3)} &= \abs{(D p, u^3)}\\
    &\le \norm{u Dp }_{\infty} \norm{u}^2\\
    &\le \frac12 \left(\norm{u \Dp p }_{\infty}
    +\norm{u \Dm   p}_{\infty}\right)\norm{u}^2
    \\
    &\le \frac12 \Big(\eps\left(\norm{\Dp u\Dm p}+ \norm{u\Dm\Dp p}+\norm{\Dp u\Dp p}
    + \norm{S^+u\Dm\Dp p}\right) \\
&\qquad      +\frac{C(\eps)}2 \big( \norm{u\Dp p} +\norm{u\Dm p} \big)
    \Big)\norm{u}^2\\
       &\le \frac12 \eps\left(\norm{\Dp u\Dm p}+\norm{\Dp u\Dp p}\right) \norm{u}^2\\
&\qquad         +\frac12 \eps\Big( \norm{u}\norm{\Dm\Dp p}_{\infty}+ \norm{S^+u}\norm{\Dm\Dp p}_\infty\\
&\qquad \qquad \qquad     +\frac{C(\eps)}2 \big( \norm{u}\norm{\Dp p}_\infty +\norm{u} \norm{\Dm p}_\infty\big)
    \Big)\norm{u}^2\\
    &\le \eps\left(\norm{\Dp u \Dp p}^2 + \norm{\Dp u \Dm p}^2\right) 
    + A(\eps,\norm{u})\\
    &\le \eps \left(\left(\left(\Dp u\right)^2, \Dp p\right)
    +\left(\left(\Dp u\right)^2,\Dm p\right)\right) + A(\eps,\norm{u_0})
  \end{align*}
  where the locally bounded function $A$ now depends on the first and second
  derivatives of $p$. Recall that $\norm{u}\le \norm{u_0}$, cf.~\eqref{eq:L2}.  Hence, 
  \begin{equation}\label{eq:wbndnew}
    \begin{aligned}
      \norm{w}_p^2 +\delta\frac{\Dx^2}{2}\norm{Du}_p^2 &\le
      \norm{u}_p^2 + A(\eps, \norm{u_0}) \Dt \\
      &\qquad + \eps \Dt \left(\left(\left(\Dp u\right)^2, \Dp
          p\right) +\left(\left(\Dp u\right)^2,\Dm p\right)\right)\\
  &\qquad + \frac{\Dx^2}{2}\left(u^2,\Dp\Dm p\right).
    \end{aligned}
  \end{equation}
  Next we study the full difference scheme by adding the ``Airy term''
  $\Dt\, \Dp^2\Dm u^{n+1}_j$. Thus the full difference scheme
  \eqref{eq:sjobergexp1} can be written
  \begin{equation*}
    v=w-\Dt\,\Dp^2\Dm v.
  \end{equation*}
  Writing this as $w=v+\Dt\,\Dp^2\Dm v$, we square it, multiply by $p$
  and sum over $j$ to get
  \begin{align*}
    \norm{w}_p^2&=\norm{v}_p^2+2\Dt \left(v,\Dp^2\Dm v\right)_p+ \Dt^2
    \norm{\Dp^2\Dm v}_p^2
    \\
    &=\norm{v}_p^2+\Dt^2\norm{\Dp^2 \Dm v}_p^2
    \\
    &\quad + 2\Dt\left( \left(\Dp v\right)^2, \Dm p\right) + \Dt
    \left(\left(\Dp v\right)^2,\Dp p\right)
    \\
    &\qquad \qquad + \Dt\Dx \norm{\Dp^2 v}_{p}^2 - \Dt
    \left(v^2 ,\Dp\Dm^2 p\right).
  \end{align*}
  Combining this with \eqref{eq:wbndnew} we get
  \begin{align*}
    \norm{v}_p^2 &+ \Dt \left(\left(\Dp v\right)^2,\Dp p\right)  +
    2\Dt\left(\left(\Dp v\right)^2,\Dm p\right)\\
    &  + \Dt^2\norm{\Dp^2 \Dm v}_p^2 + \delta\frac{\Dx^2}{2}
    \norm{Du}^2_p +
    \Dt\Dx \norm{\Dp^2 v}_{p}^2\\
    &\quad \le 
    \norm{u}_p^2 + \eps\Dt \left(\left(\Dp u\right)^2,\Dm p\right)
    + \eps\Dt  \left(\left(\Dp u\right)^2,\Dp p\right) \\
    &\qquad\qquad 
    + \Dt A(\eps,\norm{u_0}) + \Dt \left(v^2,\Dp\Dm^2 p\right).
  \end{align*}
  Rearranging and dropping some terms ``with the right sign'' we
  obtain
  \begin{align*}
    \norm{v}_p^2 + \Dt(2-\eps) &\left(\left(\Dp v\right)^2,\Dm p\right) 
    + \Dt(1-\eps)\left(\left(\Dp v\right)^2,\Dp p\right) \\
    & \le
    \norm{u}_p^2 + 2\Dt\,\eps\left(\left(\Dp u\right)^2-\left(\Dp
        v\right)^2, D p\right) \\
    &\qquad  +
    \Dt\left(A(\eps,\norm{u_0})+\frac12 \norm{u_0}^2 + \frac12
      \norm{\Dp\Dm^2p}^2\right).  
  \end{align*}
  Next, observe that 
  \begin{equation*}
    \left(\left(\Dp v\right)^2,D_\pm p\right) \ge 
    \Dx\!\!\!\!\!\!\sum_{\abs{j\Dx}\le R-1} \left(\Dp v_j\right)^2 \ge 0.
  \end{equation*}
  Define the locally bounded function $B$ by
  $B(\eps,z)=A(\eps,z)+\frac12z^2+\norm{\Dp\Dm^2p}^2$. We choose $\eps=1/2$
  and recall that $v=u^{n+1}$ and $u=u^n$. Then we get
  \begin{equation}
    \label{eq:telescope1}
    \begin{aligned}
      \norm{u^{n+1}}_{p}^2 + &2\Dt \Dx\!\!\!\!\sum_{\abs{j\Dx}\le R-1}
      \left(\Dp
        u_j^{n+1}\right)^2 \\
      &\le \norm{u^n}_p^2 + \Dt \left(\left(\Dp u^n\right)^2-\left(\Dp
          u^{n+1}\right)^2, D p\right) + \Dt B(\eps,\norm{u_0}).
    \end{aligned}
  \end{equation}
  This is a telescoping sum, and we choose $N$ such that $N\Dt = T$ to
  find
  \begin{equation}
    \label{eq:ubound1*}
    \begin{aligned}
      \norm{u^{N}}_{p}^2 + &\Dt \Dx\sum_{n=0}^{N-1}\sum_{\abs{j\Dx}\le R-1}
      \left(\Dp
        u_j^{n+1}\right)^2 \\
      &\le \norm{u^0}_p^2 + \Dt \left(\left(\Dp u^0\right)^2-\left(\Dp
          u^{N}\right)^2, D p\right) + T B(\eps,\norm{u_0}).
    \end{aligned}
  \end{equation}
  From this we can easily conclude the proof of the lemma.
\end{proof}

\begin{theorem}
  \label{theo:main_fully}
  Let $ \seq{u^n_j}$ be a sequence defined by the numerical scheme
  \eqref{eq:sjobergexp1}, and assume that there is a constant $K$ such
  that $\Dt = K\Dx^2$. Assume furthermore that $\norm{u_0}_{L^2(\R)}$
  is finite, then there exist constants $C_1$, $C_2$, and $C_3$ such
  that
  \begin{align}
     \norm{u_\Dx}_{L^{\infty}(0,T; L^2(-Q,Q))} &\le C_1, \label{eq:ubound1}\\
     \norm{u_\Dx}_{L^{2}(0,T; H^1(-Q,Q))} &\le C_2, \label{eq:ubound2}\\
     \norm{\partial_t u_{\Dx}}_{L^{{4}/{3}}(0,T; H^{-2}(-Q,Q))} &\le
     C_3,\label{eq:ubound3}
  \end{align}
  where $Q=R-1$ and $u_\Dx$ is defined by bilinear interpolation from
  $\seq{u^n_j}$, cf.~\eqref{eq:bilinearinterp}.
  Moreover, there exists a sequence of $\seq{\Dx_j}_{j=1}^\infty$ with
  $\lim_j \Dx_j =0$, and a function $u\in L^2(0,T; L^2(-Q,Q))$
  such that
  \begin{equation}\label{eq:uconverges}
     \text{$u_{\Dx_j} \to  u$ strongly in $ L^2(0,T;L^2(-Q,Q))$},
   \end{equation}
  as $j$ goes to infinity.  The function  $u$ is a weak solution of
  \eqref{eq:kdv0}.
\end{theorem}
\begin{proof}
  We first observe that $\norm{u_\Dx}\le \norm{u_0}$ so that
  \eqref{eq:ubound1} holds. To that end we first recall \eqref{eq:L2} which in particular implies that $\norm{u^{n+1}}\le \norm{u^{n}}$.  Write now
  \begin{equation*}
u_\Dx=  w_j+\frac{x-x_j}{\Dx}(w_{j+1}-w_j), \quad (x,t)\in [x_j,x_{j+1})\times [t_n,t_{n+1})
  \end{equation*}
where $w_j= u^n_j+(t-t_n)\Dtp u^n_j$. This implies
  \begin{equation*}
  \begin{aligned}
\int\abs{u_\Dx}^2\, dx&=\sum_j \int_{x_j}^{x_{j+1}} \abs{w_j+\frac{x-x_j}{\Dx}(w_{j+1}-w_j)}^2 dx \\
&=\Dx\sum_j\big(w_j^2+\frac13(w_{j+1}-w_j)^2+w_j(w_{j+1}-w_j)\big)\\
&= \frac23\norm{w}^2+\frac{\Dx}{3} \sum_j w_{j+1}w_j\\
&\le \norm{w}^2\\
&\le \norm{u^n}^2.
\end{aligned}
  \end{equation*}
The conclusion follows. 
  
  To show \eqref{eq:ubound2} we calculate that for $(x,t)\in
  [x_j,x_{j+1})\times [t_n,t_{n+1})$
  \begin{align*}
    \partial_x u_\Dx &=\Dp u^n_j + (t-t_n)\Dtp \Dp u^n_j\\
    &=\alpha_n(t) \Dp u^n_n + \left(1-\alpha_n(t)\right)\Dp u^{n+1}_j,
  \end{align*}
  where $\alpha_n(t)=(t-t_n)/\Dt\in [0,1)$. Using this, we find
  \begin{align*}
    \norm{\partial_x u_\Dx}_{L^2(0,T;L^2(-Q,Q))}^2 &=
    \int_0^T \norm{\partial_x u_\Dx(\dott,t)}_{L^2(-Q,Q)}^2 \,dt\\
    &\le 2 \sum_n\Dx \sum_{\abs{j\Dx}\le Q} \left(\Dp u^n_j\right)^2
    \frac{1}{\Dt^2}\int_{t_n}^{t_{n+1}} (t-t_n)^2\,dt \\
    &\qquad \quad \qquad + \left(\Dp u^{n+1}_j\right)^2
    \frac{1}{\Dt^2}\int_{t_n}^{t_{n+1}}
    (t_{n+1}-t)^2\,dt\\
    &\le \frac{2}{3}\Dt \sum_n \Dx\sum_{\abs{j\Dx}\le Q} \big(\left(\Dp
      u^n_j\right)^2 + \left(\Dp
      u^{n+1}_j\right)^2\big)\\
    &\le C_R,
  \end{align*}
  by Lemma~\ref{lem:h1stability}. This, and the fact that
  $\norm{u_\Dx(\dott,t)}_{L^2(-Q,Q)}\le \norm{u_0}$, proves
  \eqref{eq:ubound2}.

  Next, observe that in each cell $[x_j,x_{j+1})\times [t_n,t_{n+1})$
  \begin{equation}\label{eq:timederiv2}
    \partial_t u_{\Dx} = \Dtp u^n_j + (x -x_j) \Dp \Dtp u^n_j,
  \end{equation}
  and from the scheme we have
  \begin{equation}\label{eq:timederiv1}
    \Dtp u^n_j = \frac{\Dx^2}{2\Dt} \Dp \Dm u^n_j- \bar{u}_j^{n} D
    u_j^{n} - \Dm \Dp^{2} u_j^{n+1}. 
  \end{equation}
  We claim that for all sufficiently small $\Dx$ (actually for
  $\Dx<1/3$):
  \begin{itemize}
  \item [\bf{(a)}] For all $n\in \N_0$,
    \begin{equation*}
      \norm{\Dm \Dp^2 u^{n}}_{H^{-3}(-Q,Q)} \le C\norm {\Dp u^n}_{L^2(-Q,Q)}.
    \end{equation*}

  \item [\bf{(b)}] For all $n\in\N_0$,
    \begin{equation*}
      \norm{\Dp \Dm u^n}_{H^{-2}(-Q,Q)} \le C\norm{\Dp u^n}_{L^2(-Q,Q)}.
    \end{equation*}

  \item [\bf{(c)}] The piecewise constant function $\bar{u}_j^{n} D
    u_j^{n}$ satisfies
    \begin{equation*}
      \norm{\bar{u} D u}_{L^{4/3}(0,T,L^2(-Q,Q))} \le C,
    \end{equation*}
    for some constant which only depends on $Q$, $T$ and $u_0$.
  \end{itemize}
  To prove the first part of the claim, let $\phi \in H^3_0(-Q,Q)$ be
  any test function
  \begin{align*}
    \Bigl| \int_{-Q}^Q \left(\Dm\Dp^2 u^n\right) \phi(x)\,dx\Bigr|&=
    \Bigl|\sum_{\abs{j\Dx}\le Q} \Dm\Dp^2 u^n_j \int_{x_j}^{x_{j+1}}
    \phi(x)\,dx\Bigr|\\
    &=\Bigl|\sum_{\abs{j\Dx}\le Q} \Dp u^n_j \int_{x_j}^{x_{j+1}}
    \Dp\Dm \phi(x)\,dx\Bigr|\\
    &\le \underbrace{\sum_{\abs{j\Dx}\le Q} \abs{\Dp u^n_j}
      \int_{x_{j}}^{x_{j+1}}
      \abs{\phi''(x)}\,dx}_{\romnum{1}}\\
    &\qquad + \underbrace{\sum_{\abs{j\Dx}\le Q} \abs{\Dp u^n_j}
      \int_{x_j}^{x_{j+1}} \abs{\Dp\Dm \phi(x) - \phi''(x)}\,dx}_{\romnum{2}}.
  \end{align*}
  We start by estimating $\romnum{2}$, to that end
  \begin{align*}
    \int_{x_j}^{x_{j+1}} \abs{\Dp\Dm \phi(x) - \phi''(x)}\,dx & \le
    \frac{1}{\Dx^2}\int_{x_j}^{x_{j+1}} \int_x^{x+\Dx}\int_{z-\Dx}^{z}
    \int_x^\tau \abs{\phi'''(\theta)}\,d\theta\,d\tau\,dz\,dx\\
    &\le \frac{1}{\Dx^2}\int_{x_j}^{x_{j+1}}
    \int_x^{x+\Dx}\int_{z-\Dx}^{z}
    \sqrt{\tau-x}\norm{\phi'''}_{L^2(x,\tau)}\,d\tau\,dz \,dx\\
    &\le \frac43\Dx^{3/2}\norm{\phi'''}_{L^2(x_{j-1,x_{j+2})}}.
  \end{align*}
  Thus
  \begin{align*}
    \romnum{2}&\le \Dx\Bigl(\sum_{\abs{j\Dx}\le Q} 3\Dx \abs{\Dp
      u^n_j}^2\Bigr)^{1/2} \Bigl(\sum_{\abs{j\Dx}\le
      Q}\norm{\phi'''}_{L^2(x_{j-1},x_{j+2})}^2\Bigr)^{1/2}\\
    &\le 3\Dx \norm{\Dp u^n}_{L^2(-Q,Q)} \norm{\phi'''}_{L^2(-Q,Q)}.
  \end{align*}
  As to $\romnum{1}$, we calculate
  \begin{align*}
    \romnum{1}&=\sum_{\abs{j\Dx}\le Q} \abs{\Dp u^n_j} \int_{x_{j}}^{x_{j+1}}
    \abs{\phi''(x)}\,dx\\
    &\le \sum_{\abs{j\Dx}\le Q} \abs{\Dp u^n_j}
    \sqrt{\Dx}\norm{\phi''}_{L^2(x_j,x_{j+1})} \\
    &\le \Bigl(\sum_{\abs{j}\le Q} \Dx \abs{\Dp u^n_j}^2\Bigr)^{1/2}
    \Bigl(\sum_{\abs{j\Dx}\le
      Q}\norm{\phi''}_{L^2(x_j,x_{j+1})}^2\Bigr)^{1/2}\\
    &=\norm{\Dp u^n}_{L^2(-Q,Q)} \norm{\phi''}_{L^2(-Q,Q)}.
  \end{align*}
  Therefore {\bf (a)} follows.  Claim {\bf (b)} is proved similarly.

  To prove {\bf (c)} we first define the cut-off function $\eta$ as
  \begin{equation*}
    \eta(x) =
    \begin{cases}
      1 &\abs{x}\le Q,\\
      0 &\abs{x}\ge Q+1,\\
      x+Q+1 & x\in [-(Q+1),-Q],\\
      Q+1-x & x\in [Q,Q+1],
    \end{cases}
  \end{equation*}
  and set $\eta_j=\eta(x_j)$. Then we have that
  \begin{align*}
    \Dt \sum_{n=0}^{N-1}\Bigl(\sum_j \abs{\eta_j \bar{u}^n_j D
      u^n_j}^2 \Bigr)^{2/3} &\le \Dt \sum_{n=0}^{N-1} \norm{\eta
      u^n}_\infty^{4/3} \Bigl(
    \sum_{\abs{j\Dx}\le R} \left(D u^n_j\right)^2\Bigr)^{2/3}\\
    &\le \Bigl(\Dt \sum_{n=0}^{N-1} \norm{\eta u^n}_\infty^4\Bigr)^{1/3}
    \Bigl(\Dt\Dx\sum_{n=0}^{N-1}\sum_{\abs{j\Dx}\le R} \left(D
      u^n_j\right)^2\Bigr)^{2/3}\\
    &\le \Bigl(\Dt \sum_{n=0}^{N-1} \norm{\eta u^n}_\infty^4\Bigr)^{1/3}
    \Bigl(\Dt\Dx\sum_{n=0}^{N-1}\sum_{\abs{j\Dx}\le R} \left(\Dp
      u^n_j\right)^2\Bigr)^{2/3}\\
    &\le \Bigl(\Dt \sum_{n=0}^{N-1} \norm{\eta u^n}_\infty^4\Bigr)^{1/3}
    C_R^{2/3},
  \end{align*}
  by Lemma~\ref{lem:h1stability}. To proceed we use the inequality
  \begin{equation*}
    \norm{v}_\infty \le 2 \Bigl(\Dx\sum_{\abs{j\Dx}\le R} v_j^2\Bigr)^{1/4}
    \Bigl(\Dx\sum_{\abs{j\Dx}\le R} \left(\Dp v_j\right)^2\Bigr)^{1/4},
  \end{equation*}
  which holds for any grid function $v$ such that $v_j=0$ for
  $\abs{j\Dx}\ge R$. This can be shown as follows:
 \begin{equation*}
 \begin{aligned}
 v_j^2&=\sum_{k=-\infty}^{j-1}(v_{k+1}^2-v_{k}^2)
 =\Dx\sum_{k=-\infty}^{j-1}(v_k+v_{k+1})\frac{v_{k+1}-v_{k}}{\Dx} \\
 &\le\Big(\Dx \sum_{k=-\infty}^{j-1}(v_k+v_{k-1})^2\Big)^{1/2} 
 \Big(\Dx \sum_{k=-\infty}^{j-1}(\Dp v_k)^2\Big)^{1/2} \\
&\le \sqrt{2} \big(\norm{v}^2+\norm{v}^2\big)^{1/2} \norm{\Dp v}\\
&\le 2\norm{v}\norm{\Dp v},
 \end{aligned}
 \end{equation*} 
 which implies that
 \begin{equation*} 
 \norm{v}_\infty\le 2\norm{v}^{1/2}\norm{\Dp v}^{1/2}.
 \end{equation*}   
  
  We shall use this for $v=\eta u^n$, to that end
  observe
  \begin{align*}
    \Dx\sum_{\abs{j\Dx}\le R} \left(\eta_j u^n_j\right)^2 &\le
    \norm{u^n}^2,\\
    \Dx \sum_{\abs{j\Dx}\le R} \left(\Dp \eta_j u^n_j\right)^2 &\le
    2\Dx \sum_{\abs{j\Dx}\le R}\Big( \left(u^n_j\Dp \eta_j\right)^2 +
    \left(\eta_{j+1}\Dp u^n_j\right)^2\Big)\\
    &\le 2 \norm{u^n}^2 + 2\Dx \sum_{\abs{j\Dx}\le R} \left(\Dp
      u^n_j\right)^2.
  \end{align*}
  Hence
  \begin{equation*}
    \norm{\eta u^n}_\infty^4 \le C\norm{u^n}^2\Bigl(\norm{u^n}^2 +
    \Dx\sum_{\abs{j\Dx}\le R} \left(\Dp u^n_j\right)^2 \Bigr).
  \end{equation*}
  Thus
  \begin{align*}
    \Dt \sum_{n=0}^{N-1}\Bigl(\Dx\sum_{\abs{j\Dx}\le Q} \abs{\bar{u}^n_j
      D u^n_j}^2 \Bigr)^{2/3} &\le \Dt \sum_{n=0}^{N-1}\Bigl(\Dx\sum_j
    \abs{\eta_j \bar{u}^n_j D u^n_j}^2
    \Bigr)^{2/3} \\
    &\le C_R\biggl(\Dt\sum_n\norm{u^n}^2\Bigl(\norm{u^n}^2 +
    \Dx\!\!\!\sum_{\abs{j\Dx}\le R}\!\!\! \left(\Dp u^n_j\right)^2
    \Bigr)\biggr)^{1/3} \\
    &\le C,
  \end{align*}
  for a constant $C$ depending only on $\norm{u_0}$, $R$, and $T$. This
  proves {\bf (c)}.

  Now {\bf (a)}, {\bf (b)} and Lemma~\ref{lem:h1stability} mean that
  \begin{align*}
    \norm{\Dm\Dp^2 u^n}_{L^{4/3}(0,T,H^{-3}(-Q,Q))}^{4/3} &=
    \Dt \sum_{n=0}^N \norm{\Dm\Dp^2 u^n}_{H^{-3}(-Q,Q)}^{4/3} \\
    &\le CT^{1/3} \Bigl(\Dt\sum_{n=0}^N \norm{\Dp
      u^n}_{L^2(-Q,Q)}^2\Bigr)^{2/3}\\
    &\le C,
  \end{align*}
  and that
  \begin{equation*}
    \norm{\Dm\Dp u^n}_{L^{4/3}(0,T,H^{-3}(-Q,Q))}^{4/3} \le C.
  \end{equation*}
  Similarly, $\bar{u}^nDu^n\in L^{4/3}(0,T,L^2(-Q,Q))\subset
  L^{4/3}(0,T,H^{-3}(-Q,Q))$. Therefore, by \eqref{eq:timederiv1},
  $\Dtp u^n_j\in L^{4/3}(0,T,H^{-3}(-Q,Q))$.  Next, let
  \begin{equation*}
    \alpha(x)=\frac{1}{\Dx}\sum_j(x-x_j)\chi_{[x_j,x_{j+1})}(x).
  \end{equation*}
  Then \eqref{eq:timederiv2} reads
  \begin{equation*}
    \partial_t u_{\Dx} = \alpha \Dtp u^n + (1-\alpha)\Dtp S^+u^n.
  \end{equation*}
  Therefore
  \begin{equation*}
    \begin{aligned}
      \norm{\partial_t u_\Dx}_{L^{4/3}(0,T,H^{-3}(-Q,Q))} &\le
      \norm{\alpha}_{L^\infty(\R)} \norm{\Dtp
        u^n}_{L^{4/3}(0,T,H^{-3}(-Q,Q))}\\
      &\qquad + \norm{1-\alpha}_{L^\infty(\R)} \norm{\Dtp
        S^+u^n}_{L^{4/3}(0,T,H^{-3}(-Q,Q))}\\
      &\le 2 \norm{\Dtp u^n}_{L^{4/3}(0,T,H^{-3}(-Q,Q))} \le C,
    \end{aligned}
  \end{equation*}
  which is \eqref{eq:ubound3}.

  Using \eqref{eq:ubound1}, \eqref{eq:ubound2}, and \eqref{eq:ubound3}
  we can apply the Aubin--Simon compactness lemma (see Lemma  \ref{lemma:simon}) to conclude that
  $u_\Dx$ has a  subsequence which convergences strongly in
  $L^2(0,T;L^2(-Q,Q))$,
  i.e., \eqref{eq:uconverges} holds.

  Note that this is enough to pass to the limit in the
  nonlinearity. This means we can apply the Lax--Wendroff like result
  of \cite{HoldenKarlsenRisebro:1999} to conclude that the limit is a weak
  solution.
\end{proof}

\begin{lemma}[Aubin--Simon]
  \label{lemma:simon}
  Let $X, B, Y$ are three Banach spaces such that $X \subset B$ with
  compact embedding and $B \subset Y$ with continuous embedding. Let
  $T > 0$ and $\{u_n\}_{n \in \N}$ be a sequence such that $\{u_n\}_{n \in
    \N}$ is bounded in $L^p(0,T;X)$ and $\{\partial_t u_n\}_{n\in \N}$
  is bounded in $L^{q}(0,T;Y)$, for any $1 \le p,\,q \le \infty$.
  Then there exists $u \in L^p(0,T;B)$ such that, up to a subsequence,
  \begin{equation*}
    u_n \rightarrow u \quad \text {in} \quad L^p(0,T;B).
  \end{equation*}
\end{lemma}

\section{Numerical examples}\label{sec:numex}
We have tested the scheme for two examples where the solution is
known explicitly, and for one example where the solution is not known, but the
initial data has a singularity, and is in $L^2$.

\subsection{A one-soliton solution}
The KdV equation \eqref{eq:kdv0} has an exact solution given by 
\begin{equation}
  \label{eq:onesol}
  w_1(x,t)=9\left(1-\tanh^2\left(\sqrt{3/2}(x-3t)\right)\right).
\end{equation}
This represents a single bump moving to the right with speed 3. We have
tested our scheme with initial data $u_0(x)=w_1(x,-1)$ in order to check
how fast this scheme converges. In Figure~\ref{fig:1} we show the
exact solution at $t=2$ as well as the numerical solution computed using $1000$
grid points in the interval $[-10,10]$, i.e., $\Dx=20/1000$.
\begin{figure}[h]
  \centering
  \includegraphics[width=0.99\linewidth]{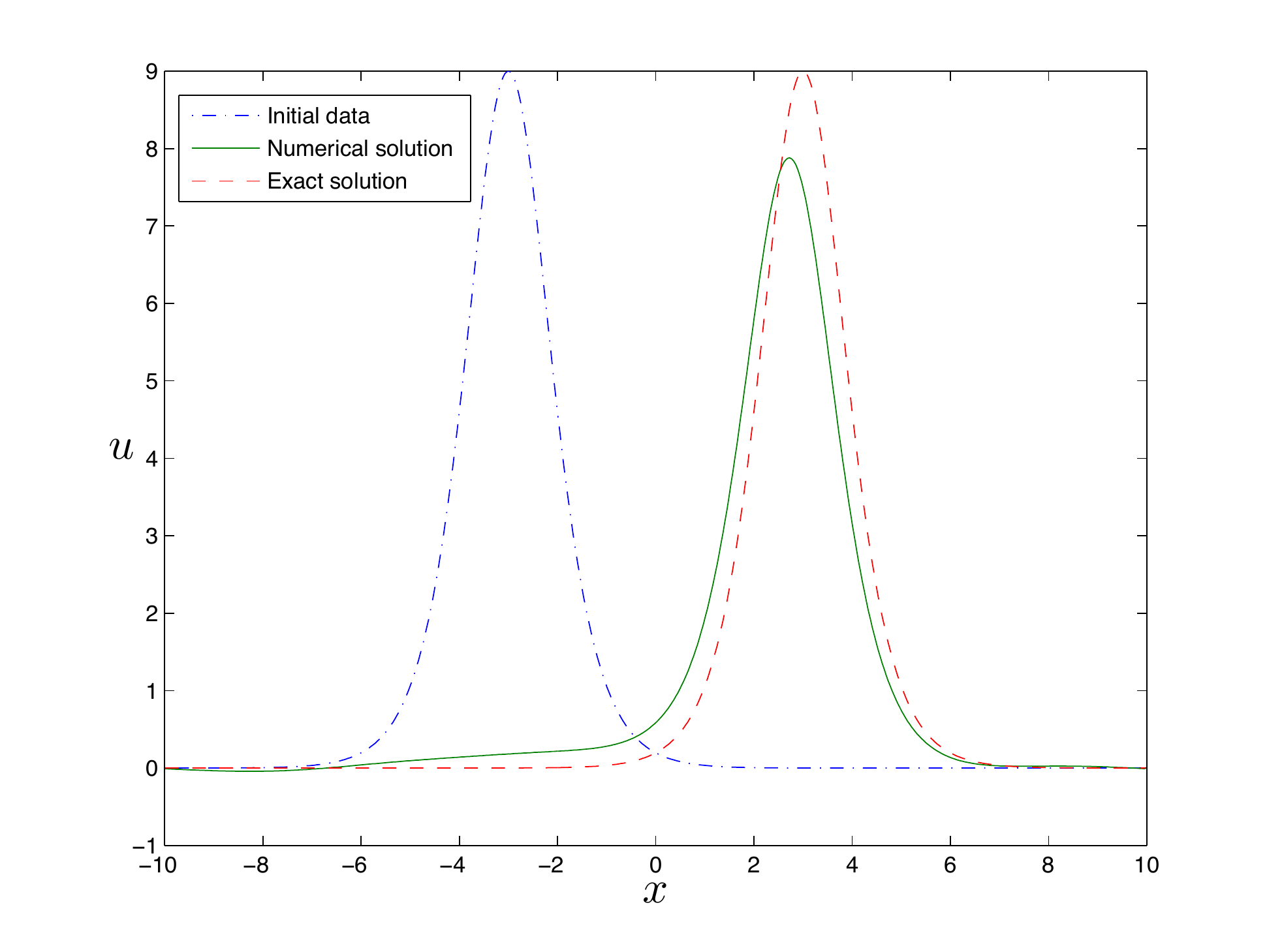}
  \caption{Initial data, and the exact and numerical solutions at
    $t=2$ with initial data $w_1(x,-1)$ with $N=1000$ grid points.}
  \label{fig:1}
\end{figure}
We have also computed numerically the error for a range of $\Dx$, where
the relative error is defined by 
\begin{equation*}
  E=100\frac{\sum_{j=1}^N \abs{w_1(x_j,1)-u_\Dx(x_j,2)}}{\sum_{j=1}^N w_1(x_j,1)}.
\end{equation*}
Recall that we are using $w_1(x,-1)$ as initial data, so that
$w_1(x,1)$ represents the solution at $t=2$.
In Table~\ref{tab:1} we show the relative errors as well as the
numerical convergence rates for this example.
\begin{table}[h]
  \centering
  \begin{tabular}[h]{c|r r}
    $N$ &\multicolumn{1}{c}{$E$} &\multicolumn{1}{c}{rate} \\
    \hline\\[-2ex]
    500  & 51.2 & \\[-1ex]
    1000 & 31.4 & \raisebox{1.5ex}{0.70} \\[-1ex]
    2000 & 17.6 & \raisebox{1.5ex}{0.83} \\[-1ex]
    4000 & 9.4  & \raisebox{1.5ex}{0.91} \\[-1ex]
    8000 & 4.9  & \raisebox{1.5ex}{0.95} \\[-1ex]
   16000 & 2.5  & \raisebox{1.5ex}{0.96}  
  \end{tabular}
  \caption{Relative errors for the one-soliton solution.}
  \label{tab:1}
\end{table}
The numerical convergence rate indicates that, as expected, the scheme
is of first order. Note also that we have to use a rather small $\Dx$
in order to get a reasonably small error. Computing soliton solutions
is quite hard, since these solutions are close to zero outside a
bounded interval, and the speed of the soliton is proportional to its
height. Therefore, if a numerical method (due to, e.g., numerical
diffusion) does not have the correct height, it will also have a wrong
speed. Thus after some time, it will be in the wrong place and the
error is close to 100\%.

\subsection{A two-soliton solution}\label{subsub:twosol}
Another exact solution of \eqref{eq:kdv0} is the so-called
two-soliton,
\begin{equation}
  \label{eq:twosoliton}
  w_2(x,t)=6(b-a)\frac{b\csch^2\left(\sqrt{b/2}(x-2bt)\right) 
    +a\sech^2\left(\sqrt{a/2}(x-2at)\right)}
  {\left(\sqrt{a}\tanh\left(\sqrt{a/2}(x-2at)\right) - 
      \sqrt{b}\coth\left(\sqrt{b/2}(x-2bt)\right)\right)^2},
\end{equation}
for any real numbers $a$ and $b$. We have used $a=0.5$ and $b=1$. This
solution represents two waves that ``collide'' at $t=0$ and separate
for $t>0$. For large $\abs{t}$, $w_2(\dott,t)$ is close to a sum of
two one-solitons at different locations. 

Computationally, this is a much harder problem than the one-soliton
solution. As initial data we have used $u_0(x)=w_2(x,-10)$. In
Figure~\ref{fig:2} we show the exact and numerical solutions at $t=20$.
\begin{figure}[h]
  \centering
  \includegraphics[width=0.99\linewidth]{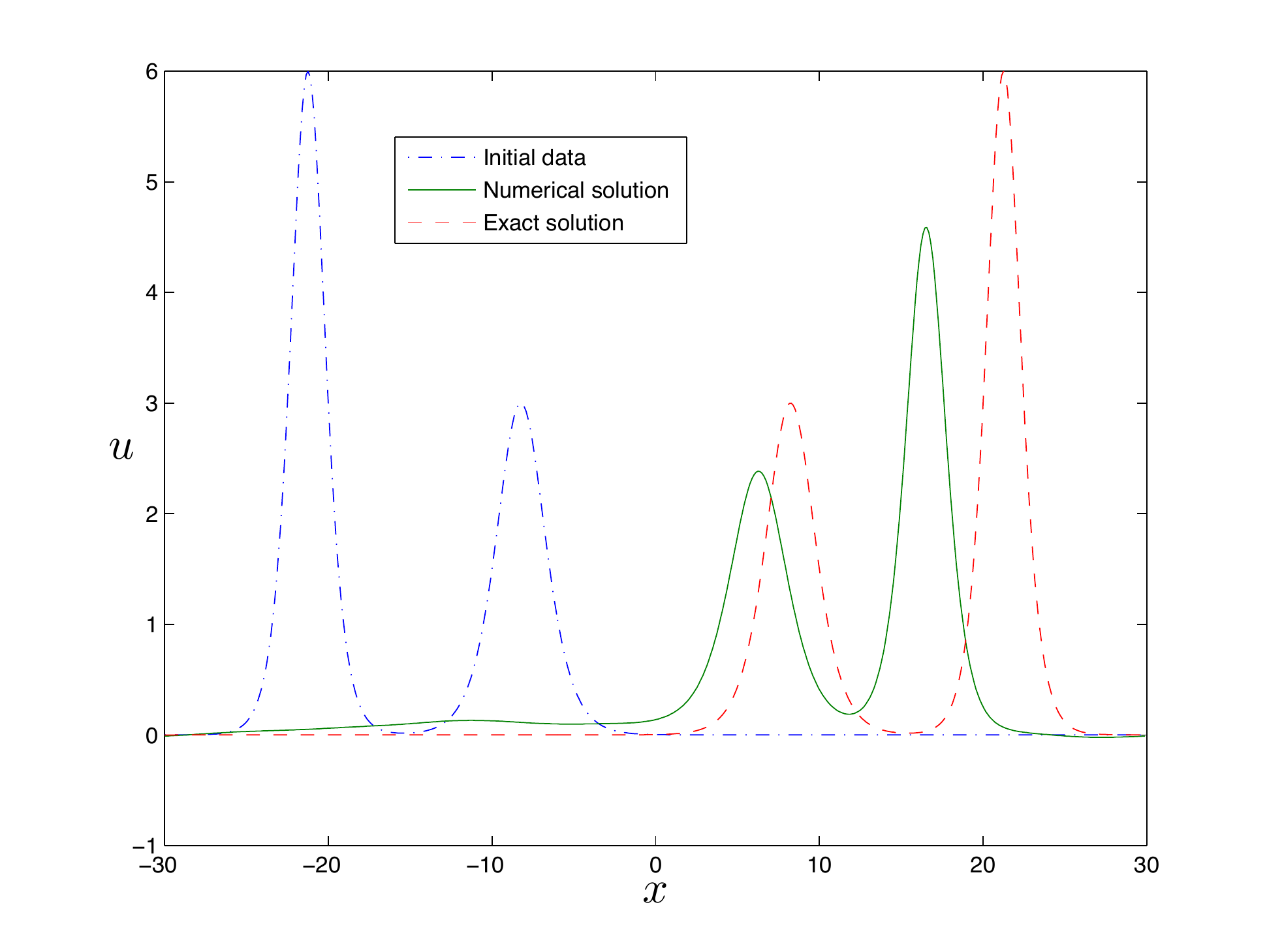}
  \caption{Initial data, and the exact and numerical solutions at
    $t=20$ with initial data $w_2(x,-10)$ with $N=4000$ grid points.}
  \label{fig:2}
\end{figure}
Although we have used 4000 grid points, the error is a staggering
140\%! We see that the qualitative features are ``right'', in the
sense that the larger soliton has overtaken the slower one, but neither
their heights nor their positions are correct. For sufficiently small
$\Dx$, the numerical solution will be close to the exact also in this
case, but it is impractical to calculate numerical convergence rates
since the computations would take too much time. 

\subsection{Initial data in $L^2$}\label{subsub:l2}
We have also tried our scheme on an example where the initial data is
in $L^2$, but not in any Sobolev space with positive
index. Furthermore, note that all the conclusions in
Section~\ref{sec:L2} remain valid if we restrict ourselves to the
periodic case. Therefore we have chosen initial data
\begin{equation}
  \label{eq:l2init}
  u_0(x)=
  \begin{cases}
    0 & x\le 0,\\
    x^{-1/3} &0<x<1,\\
    0 &x\ge 1,
  \end{cases}
\end{equation}
if $x$ is in $[-5,5]$, and extended it periodically outside this
interval. In this case we have no exact solution available. Therefore
we can only determine the convergence by viewing solutions with different
$\Dx$. In Figure~\ref{fig:3} we have plotted the numerical solutions
at $t=0.5$ using
$3750$, $7500$, $15000$ and $30000$ grid cells in the interval
$[-5,5]$. 
\begin{figure}[h]
  \centering
  \includegraphics[width=0.99\linewidth]{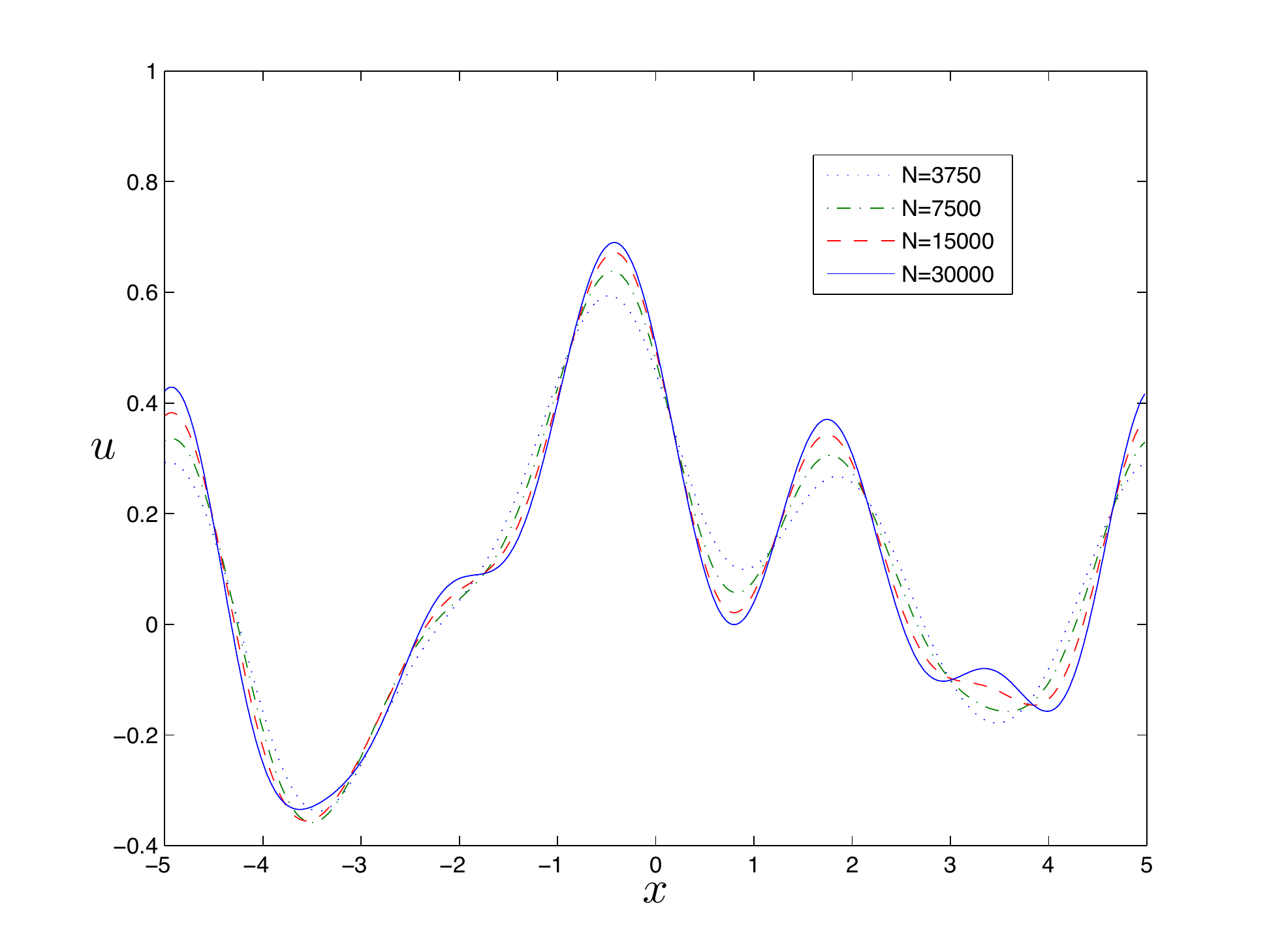}
  \caption{The numerical solution $u_\Dx(x,0.5)$ with initial data
    \eqref{eq:l2init} for various $\Dx$.}
  \label{fig:3}
\end{figure}
From this figure we can observe that the numerical solutions seem to
converge nicely to a (smooth) function. The coarser features are
already resolved using $3750$ grid cells, and only the finer
structures become more apparent for smaller $\Dx$.

\appendix
\section{Sobolev inequalities}

For the convenience of the reader we include proofs of the discrete Sobolev inequalities, that are frequently used, but rarely proved.

\begin{lemma}\label{lem:SobolevDisc} Given $m_1,m_2,m_3,n_1,n_2\in \N_0$, we define $m=m_1+m_2+m_3$ and $n=n_1+n_2$.  Assume $m<n$. Consider  $u\in \ell^2(\R)$. Given a positive $\eps$. The following estimates hold
\begin{align}
  \label{eq:sob1D}
  \norm{\Dp^{m_1}\Dm^{m_2} D^{m_3} u}_2^2&\le \eps  \norm{\Dp^{n_1}\Dm^{n_2} u}_2^2+C(\eps )\norm{u}_2^2,\\
  \label{eq:sob2aD}
  \norm{\Dp^{m_1}\Dm^{m_2} D^{m_3} u}_\infty^2&\le \eps  \norm{\Dp^{n_1}\Dm^{n_2} u}_2^2
+C(\eps )\norm{u}_2^2,
\end{align} 
for some function $C(\eps)$. 

The same estimates hold in the periodic case where $u\in\ell^\infty(\R)$ is such that there exists a period $J\in\N$ such that $u_{j+J}=u_j$ for all $j\in\Z$, and the norms are taken over the period.
\end{lemma}
\begin{proof} We here
  treat the case of the full line only. Assume first that $m_3=0$.
  The proof follows by induction. Let $m=1$ and $n=2$. Then we have
\begin{equation*}
\begin{aligned}
\norm{\Dpm u}_2^2&= -(u, \Dp\Dm u) \\
&\le\eps \norm{\Dp\Dm u}_2^2+ C(\eps) \norm{u}_2^2.
\end{aligned}
\end{equation*}  
Since $\norm{\Dp\Dm u}_2=\norm{\Dp^2 u}_2=\norm{\Dm^2 u}_2$, we have shown \eqref{eq:sob1D} in the case with $m=1$ and $n=2$.   Assume now that \eqref{eq:sob1D} holds for all cases with $m\le N$ for some fixed but arbitrary $N$,  and all $n=m+1$. Given $m_1,m_2, n_1,n_2$ such that $m=m_1+m_2=N+1$ and $n=n_1+n_2=m+1$. We then find
\begin{equation*}
\begin{aligned}
\norm{\Dp^{m_1}\Dm^{m_2} u}_2^2&
= \abs{(\Dp^{n_1}\Dm^{n_2}u, \Dp^{m_1-(n_2-m_2)}\Dm^{m_2-(n_1-m_1)} u)} \\
&\le\eps \norm{\Dp^{n_1}\Dm^{n_2} u}_2^2+ C(\eps) \norm{\Dp^{m-n_2}\Dm^{m-n_1}u}_2^2 \\
&\le \eps \norm{\Dp^{n_1}\Dm^{n_2} u}_2^2+ C(\eps)\big(\eps_1\norm{\Dp^{m_1}\Dm^{m_2} u}_2^2+C(\eps_1) \norm{u}_2^2 \big),
\end{aligned}
\end{equation*}    
using the induction hypothesis since $m-n_2+m-n_1=m+(m-n)=N$ and $m=N+1$.  We can rewrite this as
\begin{equation*}
(1-C(\eps)\eps_1)\norm{\Dp^{m_1}\Dm^{m_2} u}_2^2\le \eps \norm{\Dp^{n_1}\Dm^{n_2} u}_2^2+ C(\eps)C(\eps_1) \norm{u}_2^2.
\end{equation*}    
Given $\eps$ we choose $\eps_1$ such that $C(\eps)\eps_1\le\frac12$, which proves the case with $m=N+1$ and $n=m+1$.  By induction we have shown  \eqref{eq:sob1D} in all cases where $n=m+1$. 

Next we show how to extend this result to $n=m+2$ for arbitrary $m$. The general case of $n>m$ follows similarly.  Let now $n_1, n_2$ be such that $n=n_1+n_2=m+2$. We now have
\begin{equation*}
\begin{aligned}
\norm{\Dp^{m_1}\Dm^{m_2} u}_2^2&
\le \eps \norm{\Dp^{n_1-1}\Dm^{n_2} u}_2^2+ C(\eps)\norm{u}_2^2 \\
&\le \eps\big(\eps_1 \norm{\Dp^{n_1}\Dm^{n_2} u}_2^2+C(\eps_1)\norm{u}_2^2\big)+ C(\eps)\norm{u}_2^2 \\
&=\eps\eps_1 \norm{\Dp^{n_1}\Dm^{n_2} u}_2^2+\big(\eps C(\eps_1)+ C(\eps)\big)\norm{u}_2^2, 
\end{aligned}
\end{equation*}    
using first that $n_1-1+n_2=n-1=m+1$. This proves  \eqref{eq:sob1D} in the general case with $m_3=0$.

For an arbitrary $m_3\in\N$ we observe that
\begin{equation}
\Dp^{m_1}\Dm^{m_2}D^{m_3}=2^{-m_3} \sum_{k=0}^{m_3} \binom{m_3}{k} \Dp^{m_1+k}\Dm^{m_2+m_3-k}
  \end{equation}
using $D=\frac12(\Dp+\Dm)$, which reduces this case to that with $m_3=0$. 
 
Consider now the inequality \eqref{eq:sob2aD}. Observe that
\begin{equation*}
 u_j^2=\Dx\sum_{k=-\infty}^{j-1}\Dp u^2_k = \Dx \sum_{k=-\infty}^{j-1} (u_k+u_{k+1}) \Dp u_k \end{equation*}    
which implies that
\begin{equation*}
 \norm{u}_\infty^2= \abs{\big((u+S^+u),\Dp u\big)}\le 
\eps \norm{\Dp u}_2^2+C(\eps)\norm{u}_2^2=
\eps \norm{\Dpm u}_2^2+C(\eps)\norm{u}_2^2. 
\end{equation*}    
Thus
\begin{equation*}
\begin{aligned}
\norm{\Dp^{m_1}\Dm^{m_2} u}_\infty^2&
\le \eps \norm{\Dp(\Dp^{m_1}\Dm^{m_2} u)}_2^2+ C(\eps)\norm{\Dp^{m_1}\Dm^{m_2} u}_2^2 \\
&\le \eps (\eps_1\norm{\Dp^{n_1}\Dm^{n_2} u)}_2^2+C(\eps_1)\norm{u}_2^2)\\
&\quad+ 
C(\eps)(\eps_1\norm{\Dp^{n_1}\Dm^{n_2} u}_2^2+ C(\eps_1)\norm{u}_2^2) \\
&\le (\eps+C(\eps)) \eps_1\norm{\Dp^{n_1}\Dm^{n_2} u}_2^2+(\eps+C(\eps))C(\eps_1)\norm{u}_2^2.
\end{aligned}
\end{equation*}    
(In the rare case that $n_1=m_1+1$ and $n_2=m_2$ we do not change the first term.) Given $\eps$ we choose 
$\eps_1$ such that $(\eps+C(\eps)) \eps_1\le\eps$. This completes the proof of \eqref{eq:sob2aD}.

The proof of \eqref{eq:sob2aD} requires some modifications in the periodic case. Let $m$ be such that $\abs{u_m}=\min_{j=0,\dots,J-1}\abs{u_j}$. 
For $j>m$ we have
\begin{equation*}
u_j^2=u_m^2+\Dx\sum_{k=m}^{j-1} \Dp u_k^2=u_m^2+\Dx\sum_{k=m}^{j-1}(u_k+u_{k+1}) \Dp u_k.
\end{equation*}    
Thus
\begin{equation*}
\begin{aligned}
\max_j\abs{u_j}^2&\le \min_j\abs{u_j}^2+\abs{\big((u+u^+),\Dp u\big)} \\
&\le \frac1L\norm{u}^2_2+ \eps \norm{\Dp u}_2^2+ \tilde C(\eps) \norm{u}_2^2\\
&\le \eps \norm{\Dpm u}_2^2+ C(\eps) \norm{u}_2^2.
\end{aligned}
\end{equation*}    
Here we have used that
\begin{equation*}
J\Dx\min_j\abs{u_j}^2\le \Dx\sum_{k=0}^{J-1}\abs{u_k}^2=\norm{u}_2^2,
\end{equation*}  
which implies
\begin{equation*}
 \min_j\abs{u_j}^2\le \frac1L\norm{u}_2^2, 
 \end{equation*}  
 where $L=J\Dx$ remains finite and nonzero as $\Dx\to0$.
 \end{proof}

\begin{remark}
We could equally well have written the inequalities \eqref{eq:sob1D}, \eqref{eq:sob2aD} as
\begin{align}
  \label{eq:sob1D*}
  \norm{\Dp^{m_1}\Dm^{m_2} D^{m_3} u}_2&\le \eps  \norm{\Dp^{n_1}\Dm^{n_2} u}_2+C(\eps )\norm{u}_2,\\
  \label{eq:sob2aD*}
  \norm{\Dp^{m_1}\Dm^{m_2} D^{m_3} u}_\infty&\le \eps  \norm{\Dp^{n_1}\Dm^{n_2} u}_2+C(\eps )\norm{u}_2.
\end{align} 
\end{remark}

\noindent {\bf Acknowledgments.}
This paper was initiated during the international research program on Nonlinear Partial Differential Equations at the Centre for Advanced Study  at the Norwegian Academy of Science and Letters in Oslo  and completed during a stay at ETH in Z\"urich. Both institutions are thanked for their hospitality.


\end{document}